\documentclass[10pt]{amsart}
\usepackage{tikz}
\usetikzlibrary{matrix,arrows,decorations.pathmorphing}
\usepackage{amssymb}

\newtheorem{proposition}{Proposition}[section]
\newtheorem{lemma}[proposition]{Lemma}
\newtheorem{corollary}[proposition]{Corollary}
\newtheorem{theorem}[proposition]{Theorem}
\newtheorem{remark}[proposition]{Remark}

\theoremstyle{definition}

\newcommand{\selabel}[1]{\label{se:#1}}
\newcommand{\seref}[1]{Section~\ref{se:#1}}

\def\a{\alpha}
\def\b{\beta}

\def\D{\Delta}

\def\ep{\varepsilon}

\def\l{\lambda}

\def\ol{\overline}

\def\ot{\otimes}

\def\O{\Omega}

\def\oo{\infty}

\def\<{\leqslant}
\def\>{\geqslant}

\begin{document}
\title{Representations of the small quasi-quantum group}
\author{Hua Sun}
\address{School of Mathematical Science, Yangzhou University, Yangzhou 225002, China}
\email{huasun@yzu.edu.cn}
\author{Hui-Xiang Chen}
\address{School of Mathematical Science, Yangzhou University, Yangzhou 225002, China}
\email{hxchen@yzu.edu.cn}
\author{Yinhuo Zhang}
\address{Department of Mathematics $\&$  Statistics, University of Hasselt, Universitaire Campus, 3590 Diepeenbeek, Belgium}
\email{yinhuo.zhang@uhasselt.be}
\begin{abstract}
In this paper, we study the representation theory of the small quantum group $\overline{U}_q$ and the small quasi-quantum group $\widetilde{U}_q$,
where $q$ is a primitive $n$-th root of unity and $n>2$ is odd. All finite dimensional indecomposable $\widetilde{U}_q$-modules are described and classified. Moreover, the decomposition rules for the tensor products of $\widetilde{U}_q$-modules are given. Finally, we describe the structures of the projective class ring $r_p(\widetilde{U}_q)$ and the Green ring $r(\widetilde{U}_q)$. We show that $r(\overline{U}_q)$ is isomorphic to a subring of $r(\widetilde{U}_q)$, and the stable Green rings $r_{st}(\widetilde{U}_q)$ and $r_{st}(\overline{U}_q)$ are isomorphic.
\end{abstract}
\maketitle

\section{\bf Introduction}\selabel{1}
 In recent years, breakthroughs have been made in the classification of fusion categories of low ranks. For example, Ostrik classified the fusion categories of rank 2  in \cite{Ostrik03-2}, and the pivotal fusion categories of rank 3 in \cite{Ostrik03-3}. Rowell, Strong and Wang classified the unitary modular categories of ranks less than 4 in  \cite{E.R.Z}.
 In \cite{ChenZhang}, Chen and Zhang showed that a finite tensor category $\mathcal{C}$ (not necessarily rigid) is uniquely determined by its structure invariant data $(r(\mathcal{C}), A(\mathcal{C}), \phi, a_{i,j,l})$, where $r(\mathcal{C})$ is the Green ring of $\mathcal{C}$, $A(\mathcal{C})$ is the Auslander algebra of $\mathcal{C}$, $\phi$ is an algebra map  from $A({\mathcal C})\otimes_{\Bbbk}A({\mathcal C})$ to an algebra $M({\mathcal C})$ associated to ${\mathcal C}$, and $a_{i,j,l}$ is a family of ``invertible" matrices playing the role of the associator. Thus the classification of finite tensor categories (not necessarily rigid)  of a fixed rank can be done (in theory) by classifying the aforementioned structure invariant data.  However, in practice it is quite hard to classify all the invariant data, in particular  the associators,  which rely on  both the Green rings and the Auslander algebras. So we ask ourselves whether we can begin with a finite tensor category and determine all finite tensor categories with the same Green ring or the same Auslander algebra.  The classification of fusion categories or semisimple categories of a fixed rank is in this case with the same Auslander algebra.
 The study of 3-cocycle twisting from a Hopf algebra  to a quasi-Hopf algebra  made it possible to deform the representation category of a Hopf algebra to a non-tensor equivalent representation category of a quasi-Hopf algebra, but with the same Green ring, see \cite[Example 2.3.8]{EGNO} and \cite[Lemma 5.1]{NgP}.  This is in fact to change the associator of the representation category of the Hopf algebra by a 3-cocycle and to keep the Green ring and the Auslander algebra unchanged.  A different example of this phenomena  comes  from  the group algebras $\Bbbk Q_8$ and $\Bbbk D_8$ without 3-cocycle twisting.   The representation categories of the two group algebras are not tensor equivalent,  but they have the same Green ring and the same Auslander algebra (over an algebraically closed field).  Motivated by these examples we look at other constructions of quasi-Hopf algebras.  In \cite{YY} the authors constructed a class of finite-dimensional quasi-Hopf algebras from Andruskiewitch-Schneider's  Hopf algebras, and found  a quasi-version $\widetilde{u}_q(\mathfrak{sl}_2)$ of the small quantum group $u_q(\mathfrak{sl}_2)$,  called the small quasi-quantum group.  Since this small quasi-quantum group does not come from a direct twisting of the small quantum group, we wonder if they have the same Green ring.
This lead us to study the representation theory of the small quantum group $\overline{U}_q:=u_q(\mathfrak{sl}_2)$ and the representation theory of the small quasi-quantum group $\widetilde{U}_q:=\widetilde{u}_q(\mathfrak{sl}_2)$ for $|q|>2$ being odd.  It turns out that the two do not have the same Green ring, but have the same stable Green ring.

The paper is organized as follows. In Section 2, we first recall some concepts, the structures of $\overline{U}_q$ and $\widetilde{U}_q$. Then we give a set of central orthogonal idempotents $\{e_0, e_1, \cdots, e_{n-1}\}$ of $\widetilde{U}_q$ such that  $\sum_{i=0}^{n-1}e_i=1$. Thus, $\widetilde{U}_q$ is decomposed into a direct product of algebras $\widetilde{U}_qe_i$, i.e., $\widetilde{U}_q=\widetilde{U}_qe_0\times \widetilde{U}_qe_1\times\cdots\times \widetilde{U}_qe_{n-1}$. In Section 3, we first describe the finite dimensional indecomposable $\overline{U}_q$-modules. Then we find out all indecomposable modules over $\widetilde{U}_qe_i$ for each $0\<i\<n$ so that we can classify the indecomposable $\widetilde{U}_q$-modules. In Section 4, we investigate the tensor products of finite dimensional indecomposable $\widetilde{U}_q$-modules, and give the decomposition rules for all such tensor product modules. It is shown that the Green ring $r(\overline{U}_q)$ is isomorphic to a subring of the Green ring $r(\widetilde{U}_q)$, and that the stable Green rings $r_{st}(\widetilde{U}_q)$ and $r_{st}(\overline{U}_q)$ are isomorphic. In Section 5, we describe the structures of the Green rings $r(\overline{U}_q)$ and $r(\widetilde{U}_q)$, and the projective class ring $r_p(\widetilde{U}_q)$, respectively. It is shown that $r_p(\widetilde{U}_q)$ is generated by $2n+1$ elements subject to certain relations, but $r(\overline{U}_q)$ and $r(\widetilde{U}_q)$ are generated by infinitely many elements subject to certain relations, respectively.

\section{\bf Preliminaries}\selabel{2}

Throughout, we work over an algebraically closed field $\Bbbk$ with char$(\Bbbk)$=0. Unless
otherwise stated, all algebras and Hopf algebras are
defined over $\Bbbk$; all modules are finite dimensional and left modules;
dim and $\otimes$ denote ${\rm dim}_{\Bbbk}$ and $\otimes_{\Bbbk}$,
respectively. For basic concepts and notations about representation theory and Hopf algebras, we refer to \cite{AusReiSma, Ka, Mo}.
In particular, for a Hopf algebra, we will use $\ep$, $\D$ and $S$ to denote the counit,
comultiplication and antipode, respectively.
Let $\mathbb Z$ denote the set of all integers, $\mathbb Z^{+}$ denote the set of all positive integers
and ${\mathbb Z}_n={\mathbb Z}/n{\mathbb Z}$ for an integer $n$.
For any real number $t$, let $[t]$ be the integer part of $t$, That is, $[t]$ is maximal integer such that $[t]\<t$.

Let $0\not=q\in \Bbbk$. For any nonnegative integer $n$, define $(n)_q$ by $(0)_q=0$ and $(n)_q=1+q+\cdots +q^{n-1}$ for $n>0$.
Observe that $(n)_q=n$ when $q=1$, and
$$\begin{array}{c}
(n)_q=\frac{q^n-1}{q-1}
\end{array}$$
when $q\not= 1$.
\subsection{Green ring and stable Green ring}\selabel{2.1}
~~

For an algebra $A$, we denote by $\mathrm{Mod}A$ the category of finite dimensional $A$-modules.
For any $M\in \mathrm{Mod}A$ and any $n\in\mathbb Z$ with $n\>0$, $nM$ stand for the
direct sum of $n$ copies of $M$. Thus $nM=0$ if $n=0$. Let $P(M)$ denote the projective
cover of $M$.

The Green ring $r(H)$ of a Hopf algebra $H$ is defined to be the abelian group generated by the
isomorphism classes $[V]$ of $V$ in $\mathrm{Mod}H$
modulo the relations $[U\oplus V]=[U]+[V]$, $U, V\in\mathrm{Mod}H$. The multiplication of $r(H)$
is determined by $[U][V]=[U\ot V]$, the tensor product of $H$-modules. Then $r(H)$ is an associative ring with the identity $[\Bbbk]$. The projective class ring $r_p(H)$ of $H$ is defined to be the subring of $r(H)$ generated by
projective modules and simple modules (see \cite{Cib99}).
Notice that $r(H)$ is a free abelian group with a $\mathbb Z$-basis
$\{[V]|V\in{\rm Ind}(H)\}$, where ${\rm Ind}(H)$ denotes the category
of indecomposable objects in $\mathrm{Mod}H$.

Recall that the stable module category $\mathrm{\underline {Mod}}H$ is the quotient category of $\mathrm{Mod}H$ modulo the morphisms factoring through the projective modules. This category is triangulated with the monoidal structure derived from that of $\mathrm{Mod}H$. The Green ring of the stable module category $\mathrm{\underline {Mod}}H$, denoted $r_{st}(H)$, is called stable Green ring of $H$.
\subsection{Small quantum group $\overline{U}_q$}\label{C1.1.1}
~~
Let $q\in \Bbbk$ with $q\neq\pm 1$.
The quantum enveloping algebra  $U_q=U_q(\mathfrak{sl}_2)$ of Lie algebra $\mathfrak{sl}_2$
is generated, as an algebra, by $E$, $F$, $K$ and $K^{-1}$ subject to the following relations:
$$\begin{array}{c}
KEK^{-1}=q^2E, \ KFK^{-1}=q^{-2}F, \ [E,F]=\frac{K-K^{-1}}{q-q^{-1}}, \ KK^{-1}=K^{-1}K=1.
\end{array}$$
$U_q$ is a Hopf algebra with the coalgebra structure and the antipode given by
$$\begin{array}{c}
\bigtriangleup(K)=K\otimes K, \ \bigtriangleup(E)=E\otimes K+1\otimes E, \ \bigtriangleup(F)=F\otimes 1+K^{-1}\otimes F,\\
\varepsilon(K)=1, \ \varepsilon(E)=\varepsilon(F)=0, \ S(K)=K^{-1}, \ S(E)=-EK^{-1}, \ S(F)=-KF.\\
\end{array}$$
Note that $K$ and $K^{-1}$ are group-like elements and $E,\ F$ are skew-primitive elements.

Now let $q\in \Bbbk$ be a root of unity with the order $|q|>2$.
Set $n=|q|$ if $|q|$ is odd, and $n=\frac{|q|}{2}$ if $|q|$ is even.
Then $q^2$ is a primitive $n^{th}$ root of unity.
Let $I$ be the ideal of $U_q$ generated by $E^n$, $F^n$ and $K^n-1$.
Then $I$ is a Hopf ideal of $U_q$ and hence one gets a quotient Hopf algebra $\overline{U}_q=U_q/I$,
called the {{small quantum group}}. For the details, one can refer to \cite{Ka}.
When $|q|=2n$, the ideal $I'$ of $U_q$ generated by $E^n$, $F^n$ and $K^{2n}-1$
is also a Hopf ideal of $U_q$, and hence one can form another quotient Hopf algebra
$\widehat{U}_q=U_q/I'$. In this case, $\ol{U}_q$ is a quotient Hopf algebra of
$\widehat{U}_q$. Therefore, the category $\mathrm{Mod}\ol{U}_q$ can be regarded
as a monoidal full subcategory of the category $\mathrm{Mod}\widehat{U}_q$,
and the Green ring $r(\ol{U}_q)$ is a subring of the Green ring $r(\widehat{U}_q)$.
The finite dimensional indecomposable $\widehat{U}_q$-modules are classified in \cite{R.S}
and the decomposition rules of these modules are given in \cite{H.Y}.
The Green ring $r(\widehat{U}_q)$ is computed in \cite{SuYang}.
By \cite{H.Y}, $r(\widehat{U}_q)$ is not a commutative ring.
However, the category $\mathrm{Mod}\ol{U}_q$ is braided, and hence
the Green ring $r(\ol{U}_q)$ is commutative.
\subsection{Small quasi-quantum group $\widetilde{U}_q$}\label{C1.1.4}
Let $q\in \Bbbk$ with $q\neq\pm 1$ and let $U_q=U_q(\mathfrak{sl}_2)$ be the quantum enveloping algebra of Lie algebra $\mathfrak{sl}_2$ given as in the last subsection.
Let $n>2$ and $d$ be two positive odd integers. Let $m=nd$ and $G=\mathbb{Z}_m=\langle K\rangle$ be the cyclic group of order $m$ generated by $K$. Fix a primitive $m^{th}$-root $\xi_m$ of unity in $\Bbbk$. Let $\mathbf{m}=m^2$ and let $\xi_{\mathbf{m}}$ be a primitive $\mathbf{m}^{th}$-root of unity such that  $\xi_{\mathbf{m}}^m=\xi_m$.

In the next, let $q=\xi^d_m$. Then $q$ is a primitive $n^{th}$ root of unity. Let $\mathbb{G}=\mathbb{Z}_{\mathbf{m}}=\langle\mathbf{K}\rangle$ and $K=\mathbf{K}^m$. Let $h_1=h_2=\mathbf{K}^m\in\mathbb{G}$, $\chi_1(\mathbf{K})=\xi_{\mathbf{m}}^{2d}\in \hat{\mathbb{G}}$, and $\chi_2(\mathbf{K})=\xi_{\mathbf{m}}^{-2d}\in \hat{\mathbb{G}}$, where $\hat{\mathbb{G}}$ is the character group of $\mathbb{G}$. Let $W=\{kn|0\<k<d\}$. By \cite[Difiniteion 5.1]{YY}, for any fixed $c\in W$ and $\l=\frac{1}{q^{-1}-q}\in \Bbbk$, one can construct a quasi-Hopf algebra, called a small quasi-quantum group, we denote by $\widetilde{U}_q$. Moreover, by \cite[Proposition 5.2]{YY}, the Hopf algebra structure of $\widetilde{U}_q$ is given as follows.

For the fixed $c\in W$,
define the functions $\Phi_1, \Phi_2: G\times G\rightarrow\Bbbk^*$, $\Upsilon, \mathbf{F}_1, \mathbf{F}_2: G\rightarrow\Bbbk^*$ and $\phi_{c}: G\times G\times G\rightarrow\Bbbk^*$, respectively, by
$$\begin{array}{lll}
\Phi_1(K^s,K^r)=\xi^{cr\varphi_{11}(K^s)}_{\mathbf{m}},&
\Phi_2(K^s,K^r)=\xi^{cr\varphi_{21}(K^s)}_{\mathbf{m}},&
\Upsilon(K^s)=\xi^{-csm}_{\mathbf{m}},\\
\mathbf{F}_1(K^s)=\xi^{-c(s-2d)\varphi_{11}(K^s)}_{\mathbf{m}},&
\mathbf{F}_2(K^s)=\xi^{-c(s+2d)\varphi_{21}(K^s)}_{\mathbf{m}},&\\
\phi_{c}(K^s,K^t,K^r)=\xi^{cr[\frac{s+t}{m}]}_m.&&\\
\end{array}$$
where  $0\<s,t,r<m$, $\varphi_{11}(K^s)=\overline{(s-2d)}-(s-2d)$, $\varphi_{21}(K^s)=\overline{(s+2d)}-(s+2d)$, and $\overline{(s\pm2d)}$ is the remainder of $(s\pm2d)$ divided by $m$.

For any $g=K^{t}\in G$ with $0\<t<m$, define a character $\chi_{g}:G\rightarrow {\Bbbk}^*$ by
$$\chi_g(K^r)=\xi^{tr}_m,\ 0\<r<m,$$
and let $1_g=\frac{1}{|G|}\sum_{h\in G}\chi_{g}(h)h\in\Bbbk G$. Let
$I$ be the ideal of $U_q$ generated by $E^n$, $F^n$ and $K^{dn}-1$, and $\widetilde{U}_q:=U_q/I$ the corresponding  quotient algebra. Then $\Bbbk G\subset\widetilde{U}_q$. It is easy to see that dim$(\widetilde{U}_q)=dn^3$.  One can check that $\widetilde{U}_q=U_q/I$ is a quasi-Hopf algebra with the comultiplication, the counit given by
$$\begin{array}{ll}
\Delta(E)=\sum_{f,g\in G}\Phi_1(f,g)E1_f\otimes 1_g+K\otimes E,&\varepsilon(E)=0,\\
\Delta(F)=\sum_{f,g\in G}\Phi_2(f,g)\chi_g(K)F1_f\otimes 1_g+1\otimes F,&\varepsilon(F)=0,\\
\Delta(K)=K\otimes K,&\varepsilon(K)=1.\\
\end{array}$$
The antipode  $(S,\alpha,1)$ is given by
$\alpha=\sum_{g\in G}\Upsilon(g)1_g$, and
$$\begin{array}{l}
S(F)=\chi^{-1}(K)\sum_{g\in G}\chi_g(K)\mathbf{F}_2(g)F 1_g,\
S(E)=\sum_{g\in G}\mathbf{F}_1(g)E1_g, \ S(K)=K^{-1}.
\end{array}$$
Moreover, the associator of $\widetilde{U}_q$ is given by
$$\Phi_c=\Sigma_{f,g,h\in G}\phi_c(f,g,h)(1_f\otimes 1_g\otimes 1_h).$$

Throughout, unless otherwise stated, we assume $c=d=n>2$ to be odd, and let $H=\widetilde{U}_q$ be the small quasi-quantum groups given above. Then $m=n^2$. Let $\mathbf{q}=\xi_{m}$.  Then  $\mathbf{q}$ is a primitive $m^{th}$-root of unity and $q=\mathbf{q}^n$.

For any $0\<i\<n-1$, let $e_i=\frac{1}{n}\sum_{t=0}^{n-1}q^{ti}K^{tn}.$
Then one can easily check that $\{e_i|0\<i\<n-1\}$ forms  a set of orthogonal central idempotents, and $\sum_{i=0}^{n-1} e_i=1$. Thus $H\cong He_0\times He_1\times\cdots\times He_{n-1}$ as algebras.
Let $E_i:=Ee_i, F_i:=Fe_i, K_i:=Ke_i$ in $He_i$, $0\<i\<n$. Then we have the following Lemma.
\begin{lemma}\label{5.3.1}
For any $0\<i\<n-1$, $He_i$ is generated, as an algebra, by $E_i,F_i,K_i$ subject to following relations
$$\begin{array}{c}
K_iE_iK^{-1}_i=q^2E_i, \ K_iF_iK^{-1}_i=q^{-2}F_i, \ [E_i,F_i]=\frac{K_i-K^{-1}_i}{q-q^{-1}}, \\ K_iK^{-1}_i=K^{-1}_iK_i=e_i, \
E^n_i=F^n_i=0,K^n_i=q^{n-i}e_i,
\end{array}$$
where $e_i$ is the identity of $He_i$.
\end{lemma}
\begin{proof}
Since $H$ is generated, as an algebra, by $E, F$ and $K$, $He_i$ is generated by $E_i, F_i, K_i$. A straightforward computation shows that the relations in the lemma are satisfied in $He_i$. Hence dim$(He_i)\<n^3$. Since $H\cong He_0\times He_1\times\cdots\times He_{n-1}$ as algebras, it follows from dim$(H)=n^4$ that dim$(He_i)=n^3$ for all $0\<i\<n-1$.
Thus, the lemma follows.
\end{proof}

\begin{lemma}\label{5.3.4}
Let $0\<i\<n-1$ and $1\<r\<n-1$. Then $$\begin{array}{c}
F_iE^r_i=E^r_iF_i+E^{r-1}_i(\frac{1-q^{-2r}}{(1-q^{-2})(q-q^{-1})}K^{-1}_i-\frac{1-q^{2r}}{(1-q^2)(q-q^{-1})}K_i).
\end{array}$$
Consequently,
$$\begin{array}{c}
FE^r=E^rF+E^{r-1}(\frac{1-q^{-2r}}{(1-q^{-2})(q-q^{-1})}K^{-1}-\frac{1-q^{2r}}{(1-q^2)(q-q^{-1})}K).
\end{array}$$
\end{lemma}
\begin{proof}
It follows from Lemma \ref{5.3.1} and the induction on $r$.
\end{proof}

Denote by Ind$H$ the category of indecomposable $H$-modules. Since $H\cong He_0\times He_1\times\cdots\times He_{n-1}$ as algebras, $\mathrm{Mod}He_i$ can be naturally regarded as a full subcatefory of $\mathrm{Mod}H$. Hence the category Ind$He_i$ of indecomposable $He_i$-modules is a subcatefory of Ind$H$, $0\<i\<n-1$. Thus, we have the following corollary.
\begin{corollary}\label{5.3.2}
Let $M$ be an $H$-module.
\begin{enumerate}
	\item[(1)] $M\in$ {\rm Ind}$H$ if and only if $M\in $ {\rm Ind}$He_i$ for some $0\<i\<n-1$.
    \item[(2)] $M$ is a simple $H$-module if and only if $M$ is a simple $He_i$-module for some $0\<i\<n-1$.
\end{enumerate}
\end{corollary}

\section{\bf Indecomposable $H$-modules}\selabel{5}
In this section, we  investigate all indecomposable modules over  $H$. For an $H$-module $M$ and $x\in M$, let $\langle x\rangle$ denote the submodule of $M$ generated by $x$.

Since $|q|=n$ is odd, it follows from Lemma \ref{5.3.1} that $He_0\cong \overline{U}_q$ as algebras. Precisely, there exists an algebra isomorphism $\psi$ from $He_0$ to $\overline{U}_q$ given by
$$\psi(K_0)=K,\ \psi(E_0)=E,\ \psi(F_0)=F,\ \psi(K^{-1}_0)=K^{-1}.$$
 Hence by the discussion above, $\mathrm{Mod}\overline{U}_q$ can be regarded as a subcategory of $\mathrm{Mod}H$. Thus, each $\overline{U}_q$-module is an $H$-module.

\subsection{Indecomposable $\overline {U}_q$-modules}\label{4.1} In this subsection, we investigate the indecomposable $\overline {U}_q$-modules. Throughout this subsection, assume that $n>2$ is odd and $q\in\Bbbk$ is a root of unity with $|q|=n$ or $|q|=2n$. Therefore, $q^2$ is a primitive $n^{th}$-root of unity. Thus, one can form a Hopf algebra $H_n(1,q^2)$ (see \cite{Ch1,Ch2}). By \cite[Proposition 4.5]{Andrea},
there is a Hopf algebra epimorphism $\phi: H_n(1, q^2)\rightarrow \overline{U}_q$
determined by
$$\phi(a)=E,\phi(b)=K,\phi(c)=K,\phi(d)=q^{-2}(q-q^{-1})FK$$
with the ${\rm Ker}(\phi)=(\Bbbk C)^{+}H_n(1,q^2)$, where $C$ is the group of central group-like elements in $H_n(1,q^2)$.
Therefore, $H_n(1,q^2)/(\Bbbk C)^{+}H_n(1,q^2)\cong \overline{U}_q$. Under this isomorphism, we identify $H_n(1,q^2)/(\Bbbk C)^{+}
H_n(1,q^2)$ with $\overline{U}_q$ in the sequel.
Thus, $\mathrm{Mod}\overline{U}_q$ can be regarded as a monoidal full subcategory of $\mathrm{Mod}H_n(1,q^2)$.
Note that $C$ is the cyclic group of order $n$ generated by $bc^{-1}$.
Therefore, a $\overline{U}_q$-module is exactly an $H_n(1,q^2)$-module $M$ satisfying $(b-c)M=0$.

By \cite{Ch4}, one can get the classification of the finite dimensional indecomposable
modules over $H_n(1,q^2)$, denoted by the same symbols.

From the structures of the finite dimensional indecomposable $H_n(1, q^2)$-modules,
one can see that a finite dimensional indecomposable $H_n(1, q^2)$-module $M$ satisfies
$(b-c)M=0$ if and only if each simple factor $S$ of $M$ satisfies $(b-c)S=0$
if and only if there is  a simple factor $S$ of $M$ such that $(b-c)S=0$.
In particular,  the $H_n(1, q^2)$-module $V(l,r)$ is a $\ol{U}_q$-module if and only if $n|2r+l-1$.
Thus, we can describe all finite dimensional indecomposable $\ol{U}_q$-modules.
 We refer to \cite{Ch4} for the notations of the indecomposable modules.
Let $\oo$ be a symbol with $\oo\not\in \Bbbk$ and let $\overline \Bbbk=\Bbbk \cup\{\oo\}$.
%Let $\mathbb{Z}_2=\mathbb{Z}/2\mathbb{Z}$.

\begin{itemize}
\item Simple modules: $V_l:=V(l,\frac{3n}{4}+(-1)^{l-1}\frac{n}{4}-\frac{l-1}{2})$, $1\<l\<n$. Moreover, $V_n$ is both projective and injective.

\item Non-simple projective modules: $P_l:=P(V_l)=P(l,\frac{3n}{4}+(-1)^{l-1}\frac{n}{4}-\frac{l-1}{2})$, $1\<l\<n-1$.
Moreover, $P_l$ is also the injective envelope of $V_l$. Let $P_n=V_n$ in this case.

\item String modules: $\O^{\pm s}V_l$, $s\>1$,$1\<l\<n-1$.

\item Band modules: $M_s(l,\eta):=M_s(l,\frac{3n}{4}+(-1)^{l-1}\frac{n}{4}-\frac{l-1}{2},\eta )$, $1\<l\<n-1$, $s\geq 1$, $\eta \in \overline{\Bbbk}$.
\end{itemize}

\subsection{Simple $H$-modules}\label{5.2}
In this subsection, we investigate simple $H$-modules.
  From last subsection, we see that there are $n$ simple $He_0$-modules $V_l$, $1\<l\<n$. Hence via the algebra isomorphism $\psi$ form $He_0$ to $\overline{U}_q$, $V_l$ is an $l$-dimensional simple $H$-module with a standard $\Bbbk$-basis $\{m_1,m_2,\cdots,m_l\}$ such that the $H$-module action is determined by
\begin{equation*}
\begin{array}{ll}
Em_i=\left\{
\begin{array}{ll}
m_{i+1},& 1\<i<l,\\
0,& i=l,\\
\end{array}\right.&
Fm_i=\left\{
\begin{array}{ll}
0, & i=1,\\
\beta
_{i-1}(l)m_{i-1},& 1<i\<l,\\
\end{array}\right.\\
Km_i=q^{2i-l-1}m_i,\ 1\<i\<l,\\
\end{array}
\end{equation*}
where $\b_i(l)=\frac{\a_i(l)}{q^{2i-l}-q^{2i-l-2}}$ and $\a_{i}(l)=(i)_{q^2}(1-q^{2(i-l)})$, $1\<i\<l-1$. $V_1$ is trivial.

For any $1\<t\<n-1$ and $r\in \mathbb{Z}$, let ${\bf V}(t,r)$ be a vector space of dimension $n$ with a $\Bbbk$-basis $\{v_1,v_2,\cdots,v_n \}$. Then  a straightforward verification shows that ${\bf V}(t,r)$ is an $H$-module with the following action:
\begin{equation*}
\begin{array}{ll}
Ev_i=\left\{
\begin{array}{ll}
v_{i+1},& 1\<i<n,\\
0,& i=n,\\
\end{array}\right.&
Fv_i=\left\{
\begin{array}{ll}
0, & i=1,\\
\gamma_{i-1}v_{i-1},& 2\<i\<n,\\
\end{array}\right.\\
Kv_i=\mathbf{q}^tq^{2(r+i-1)}v_i,\ 1\<i\<n,\\
\end{array}
\end{equation*}
where $\gamma_j=\frac{1}{q-q^{-1}}(\mathbf{q}^{-t}\frac{q^{-2r}-q^{-2(r+j)}}{1-q^{-2}}-\mathbf{q}^t\frac{q^{2r}-q^{2(r+j)}}{1-q^2})\neq 0$
for $1\<j\<n-1$. Such a basis $\{v_1,v_2,\cdots,v_n \}$ is called a standard basis of ${\bf V}(t,r)$.

For any $H$-module $W$, let $W_F=\{w\in W|Fw=0\}$. If $W\neq 0$ then $W_F\neq 0$. Indeed, let $0\neq x \in W$. Then $F^nx=0$, and hence there is a positive integer $i\<n$ such that $F^ix=0$ and $F^{i-1}x\neq 0$. Thus, $F^{i-1}x\in W_F$, and so $W_F\neq 0$.

\begin{lemma}\label{5.2.2}
 Let $1\<t\<n-1$ and $r\in \mathbb{Z}$. Then ${\bf V}(t,r)$ is a simple $H$-module.
\end{lemma}
\begin{proof}
It follows from the structure of the $H$-module ${\bf V}(t,r)$ that  ${\bf V}(t,r)_{F}=\Bbbk v_1$ is $1$-dimensional. Let $W$ be a nonzero $H$-submodule of ${\bf V}(t,r)$. Then $W_{F}\neq 0$ and $W_{F}\subseteq {\bf V}(t,r)_{F}$. Hence $W_{F}={\bf V}(t,r)_{F}=\Bbbk v_1$ and $v_1\in W$. This implies $v_i\in W$ for all $1\<i\<n$ and so $W={\bf V}(t,r)$. Hence ${\bf V}(t,r)$ is a simple $H$-module.
\end{proof}

%\begin{lemma}\label{5.2.3}
%Let  $x_1,x_2,\cdots,x_n$ be $n$ distinct scales in $\Bbbk$ and $\tau$ a permutation of $\{1,2,3,\cdots,n\}$. If %there is a matrix $A\in M_n(\Bbbk)$ such that $$A^{-1}{\rm diag}\{x_1,x_2,\cdots,x_n\}A={\rm %diag}\{x_{\tau(1)},x_{\tau(2)},\cdots,x_{\tau(n)}\}.$$
%Then $A=DB$, where $D$ is a diagonal matrix and $B$ is a permutation matrix.
%\end{lemma}
%\begin{proof}
%By the definition of $\tau$, there exists a unique permutation matrix $B$ such that $$B^{-1}{\rm %diag}\{x_1,x_2,\cdots,x_n\}B={\rm diag}\{x_{\tau(1)},x_{\tau(2)},\cdots,x_{\tau(n)}\}.$$
%Hence $A^{-1}{\rm diag}\{x_1,x_2,\cdots,x_n\}A=B^{-1}{\rm diag}\{x_1,x_2,\cdots,x_n\}B$, and so $${\rm %diag}\{x_1,x_2,\cdots,x_n\}AB^{-1}=AB^{-1}{\rm diag}\{x_1,x_2,\cdots,x_n\}.$$
%Thus, $AB^{-1}=D$ is a diagonal matrix since $x_1,x_2,\cdots,x_n$ are distinct.
%\end{proof}

\begin{proposition}\label{5.2.4}
 Let $1\<t,t'\<n-1$ and $r,r'\in\mathbb Z$. Then ${\bf V}(t,r)\cong {\bf V}(t',r')$ if and only if $t=t'$ and $r\equiv r'$ $({\rm mod }\ n)$.
\end{proposition}
\begin{proof}
If $t=t'$ and $r\equiv r'$ $({\rm mod }\ n)$, then  ${\bf V}(t,r)\cong {\bf V}(t',r')$ is obvious. Conversely, assume that $\phi: {\bf V}(t,r)\rightarrow{\bf V}(t',r')$ is an $H$-module isomorphism. Then $\phi({\bf V}(t,r)_F)={\bf V}(t',r')_F$. Let $0\neq v\in {\bf V}(t,r)_F$. Then $0\neq\phi(v)\in{\bf V}(t',r')_F$. By the proof of Lemma \ref{5.2.2}, $Kv=\mathbf{q}^tq^{2r}v$ and $K\phi(v)=\mathbf{q}^{t'}q^{2r'}\phi(v)$. It follows from $\phi(Kv)=K\phi(v)$ that $\mathbf{q}^tq^{2r}=\mathbf{q}^{t'}q^{2r'}$, which implies $t=t'$ and $r\equiv r'$ $({\rm mod}\ n)$.
\end{proof}

\begin{proposition}\label{5.2.5}
Let $V$ be a simple $H$-module. Then $V$ is isomorphic to $V_l$ or ${\bf V}(t,r)$ for some $1\<l\<n$, $1\<t\<n-1$ and $r\in\mathbb Z$.
\end{proposition}
\begin{proof}
Let $V$ be a simple $H$-module, and let $\phi:$ $H\longrightarrow {\rm End}(V)$ be the corresponding algebra map. One can easily check that $\phi(K)(V_F)\subseteq V_F$. Let $v_1\in V$ be an eigenvector of $\phi(K)$. Then there exists an integer $s$ such that $Kv_1=\mathbf{q}^sv_1$ because $K^{n^2}=1$. Since $E^n=0$, there exists an integer $l$ with $1\<l\<n$ such that $E^{l-1}v_1\neq 0$ and $E^lv_1=0$. Thus, one can check that $N=$span$\{v_1,Ev_1,E^2v_1,\cdots,E^{l-1}v_1\}$ is a submodule of $V$. Therefore, $V=N$. Moreover, by Lemma \ref{5.3.4}, a straightforward computation shows that
\begin{equation*}
\begin{array}{ll}
EE^iv_1=\left\{
\begin{array}{ll}
E^{i+1}v_1,& 0\<i<l-1,\\
0,& i=l-1,\\
\end{array}\right.&
FE^iv_1=\left\{
\begin{array}{ll}
0, & i=0,\\
\theta_{i}E^{i-1}v_1,& 1\<i\<l-1,\\
\end{array}\right.\\
KE^iv_1=\mathbf{q}^sq^{2i}E^{i}v_1,\ 0\<i\<l-1,\\
\end{array}
\end{equation*}
where $\theta_i=\frac{1}{q-q^{-1}}(\mathbf{q}^{-s}\frac{q^{2i}-1}{q^{2i}}\frac{q^2}{q^2-1}-\mathbf{q}^s\frac{1-q^{2i}}{1-q^2})$ for $1\<i\<l-1$. Since $V$ is an $H$-module, $EFE^{l-1}v_1-FEE^{l-1}v_1=\frac{K-K^{-1}}{q-q^{-1}}E^{l-1}v_1$. Hence $\theta_{l-1}=\frac{1}{q-q^{-1}}(q^{2(l-1)}\mathbf{q}^s-q^{-2(l-1)}\mathbf{q}^{-s})$, and so
$$\begin{array}{c}
\frac{1}{q-q^{-1}}(\mathbf{q}^{-s}\frac{q^{2(l-1)}-1}{q^{2(l-1)}}\frac{q^2}{q^2-1}-\mathbf{q}^s\frac{1-q^{2(l-1)}}{1-q^2})=\frac{1}{q-q^{-1}}(q^{2(l-1)}\mathbf{q}^s-q^{-2(l-1)}\mathbf{q}^{-s}),
\end{array}$$
which is equivalent to $\mathbf{q}^{2s}(q^{2l}-1)=q^{2-2l}(q^{2l}-1)$.
If $1\<l<n$, then $\mathbf{q}^{2s}=q^{2-2l}$, and hence $s\equiv n(1-l)$ (mod $n^2$). In this case, one can check $V\cong V_l$. Now assume $l=n$ and let $s=nk+t$ with $0\<t\<n-1$. Since $V$ is simple, $\theta_i\neq 0$ for all $1\<i\<n-1$, or equivalently
$n^2\nmid s+n(i-1)$ for all $1\<i\<n-1$. If $t=0$, then $n\nmid k+i-1$ for all $1\<i\<n-1$, or equivalently $n|k-1$. In this case, one can check $V\cong V_n$.
If $t\neq 0$ and $k$ is even, then one can check $V\cong {\bf V}(t,\frac{k}{2})$. If $t\neq 0$ and $k$ is odd, then one can check $V\cong {\bf V}(t,\frac{k+n}{2})$.
\end{proof}

For any $r\in\mathbb Z$, we still denote by $r$ the image of $r$ under the canonical epimorphism ${\mathbb Z}\rightarrow{\mathbb Z}_n$. Summarizing the above discussion, one obtains the following corollary.

\begin{corollary}\label{5.2.6}
The set $\{V_l,{\bf V}(t,r)|1\<l\<n, 1\<t\<n-1, r\in{\mathbb Z}_n\}$ is a representative set of isomorphic classes of simple $H$-modules.
\end{corollary}

\subsection{Indecomposable $H$-modules}\label{5.3}
In this subsection, we investigate all indecomposable $H$-modules.

By Corollary \ref{5.3.2} and Subsection \ref{4.1}, it remains to classify all indecomposable $He_i$-module, $1\<i\<n-1$.
For any $1\<i\<n-1$, one can check that $\{E^j_iF^l_iK^t_i|0\<j,l,t\<n-1\}$ is a $\Bbbk$-basis of $He_i$.
For $1\<i\<n-1$ and $1\<j\<n$, let
$$\begin{array}{c}
T^i_j=\frac{1}{n}(\sum_{k=0}^{n-1}\mathbf{q}^{((j-1)n+i)k}K^{k}_i).
\end{array}$$
A tedious but routine computation shows that $\{T^i_j|1\<j\<n\}$ is a set of orthogonal idempotents of $He_i$ such that $\sum_{j=1}^nT^i_j=e_i$. Note that $He_iT^i_j=HT^i_j$.
Hence $He_i=\oplus_{j=1}^nHT^i_j$ as $H$-modules. One can easily check that $K_iT^i_j=\mathbf{q}^{(1-j)n-i}T^i_j$ and $K^{-1}_iT^i_j=\mathbf{q}^{i-(1-j)n}T^i_j$. Thus, $HT^i_j$ has a $\Bbbk$-basis $\{E^s_iF^l_iT^i_j|0\<s,l\<n-1\}$.

Let $1\<i\<n-1$ and $1\<j\<n$. For $1\<s\<k\<n$,  define $a^{ij}_{ks}\in \Bbbk$ by $a^{ij}_{kk}=1$ for $1\<k\<n$,
$$\begin{array}{c}
a^{ij}_{k1}=\frac{q^{2(k-1)}\mathbf{q}^{(1-j)n-i}-q^{2(1-k)}\mathbf{q}^{i-(1-j)n}}{q-q^{-1}}\neq 0\ \text{ for }2\<k\<n,
\end{array}$$
and $a^{ij}_{ks}=a^{ij}_{k1}+a^{ij}_{k-1,s-1}$ for $2\<s<k\<n$. It is easy to check $a^{ij}_{ks}=a^{ij}_{k-s+1,1}+a^{ij}_{k-s+2,1}+\cdots +a^{ij}_{k,1}\neq 0$.
For $1\<k\<n$, define
$$\begin{array}{c}
A^{ij}_k=\sum_{p=0}^{k-1}(\prod_{t=0}^{p}a^{ij}_{k,k-t}) E^{k-1-p}_iF^{n-1-p}_i.\\
\end{array}$$

\begin{lemma}\label{5.3.5}
Let $1\<i\<n-1$ and $1\<j, k\<n$. Then $F_iA^{ij}_kT^i_j=0$.
\end{lemma}
\begin{proof}
By Lemma \ref{5.3.4} and
$$\begin{array}{c}
A^{ij}_k=(\prod_{t=1}^{k}a^{ij}_{k,t})F^{n-k}_i+\sum_{p=0}^{k-2}(\prod_{t=0}^{p}a^{ij}_{k,k-t})E^{k-1-p}_iF^{n-1-p}_i,\\
\end{array}$$
 we have
$$\begin{array}{rl}
\vspace{0.1cm}
&F_iA^{ij}_kT^i_j\\
\vspace{0.1cm}
=&(\prod_{t=1}^{k}a^{ij}_{k,t})F^{n-k+1}_iT^i_j+\sum_{p=0}^{k-2}(\prod_{t=0}^{p}a^{ij}_{k,k-t})F_iE^{k-1-p}_iF^{n-1-p}_iT^i_j\\
\vspace{0.1cm}
=&(\prod_{t=1}^{k}a^{ij}_{k,t})F^{n-k+1}_iT^i_j+\sum_{p=0}^{k-2}(\prod_{t=0}^{p}a^{ij}_{k,k-t})E^{k-1-p}_iF^{n-p}_iT^i_j\\
\vspace{0.1cm}
&+\sum_{p=0}^{k-2}(\prod_{t=0}^{p}a^{ij}_{k,k-t})E^{k-2-p}_i\frac{1-q^{-2(k-1-p)}}{(1-q^{-2})(q-q^{-1})}K^{-1}_iF^{n-1-p}_iT^i_j\\
\vspace{0.1cm}
&-\sum_{p=0}^{k-2}(\prod_{t=0}^{p}a^{ij}_{k,k-t})E^{k-2-p}_i\frac{1-q^{2(k-1-p)}}{(1-q^2)(q-q^{-1})}K_iF^{n-1-p}_iT^i_j\\
\vspace{0.1cm}
=&\sum_{p=1}^{k-1}(\prod_{t=0}^{p}a^{ij}_{k,k-t})E^{k-1-p}_iF^{n-p}_iT^i_j\\
\vspace{0.1cm}
&+\sum_{p=0}^{k-2}(\prod_{t=0}^{p}a^{ij}_{k,k-t})E^{k-2-p}_iF^{n-1-p}_i
\frac{(1-q^{2(p+1-k)})q^{2(n-1-p)}\mathbf{q}^{i-(1-j)n}}{(1-q^{-2})(q-q^{-1})}T^i_j\\
\vspace{0.1cm}
&-\sum_{p=0}^{k-2}(\prod_{t=0}^{p}a^{ij}_{k,k-t})E^{k-2-p}_iF^{n-1-p}_i\frac{(1-q^{2(k-1-p)})q^{2(p+1-n)}\mathbf{q}^{(1-j)n-i}}{(1-q^{2})(q-q^{-1})}T^i_j\\
\vspace{0.1cm}
=&\sum_{p=1}^{k-1}(\prod_{t=0}^{p}a^{ij}_{k,k-t})E^{k-1-p}_iF^{n-p}_iT^i_j+\sum_{p=0}^{k-2}(\prod_{t=0}^{p}a^{ij}_{k,k-t})\\
\vspace{0.1cm}
&E^{k-2-p}_iF^{n-1-p}_iT^i_j\frac{\mathbf{q}^{i-(1-j+2p)n}-\mathbf{q}^{i-(2k-j-1)n}+\mathbf{q}^{(3+2p-j)n-i}-\mathbf{q}^{(1-j+2k)n-i}}{(q^2-1)(q-q^{-1})}\\
=&0,
\end{array}$$
where the last equality follows from the fact that
$$\begin{array}{c}
-a^{ij}_{k,k-p}=\frac{\mathbf{q}^{i-(2p-j-1)n}-\mathbf{q}^{i-(2k-j-1)n}+
\mathbf{q}^{(2p+1-j)n-i}-\mathbf{q}^{(1-j+2k)n-i}}{(q^2-1)(q-q^{-1})}.
\end{array}$$
\end{proof}
For any $1\<i\<n-1$ and $1\<j, k\<n$, one can check $K_iA^{ij}_k=q^{2k}A^{ij}_kK_i$.
Let $r_{k,j}:=\frac{n+2k-j}{2}$ if $2k-j$ is odd, and  $r_{k,j}:=\frac{2k-j}{2}$ if $2k-j$ is even.  Then we have the following proposition.

\begin{proposition}\label{5.3.6}
Let $1\<i\<n-1$ and $1\<j, k\<n$. Then the $He_i$-submodule $\langle A^{ij}_kT^i_j\rangle$ of $HT^i_j$ is isomorphic to ${\bf V}(n-i,r_{k,j})$, where $r_{k,j}$ is defined as above.
\end{proposition}
\begin{proof}
One can easily check that
$$\{A^{ij}_kT^i_j, E_iA^{ij}_kT^i_j, E^2_iA^{ij}_kT^i_j, \cdots, E^{n-1}_iA^{ij}_kT^i_j\}$$ is a $\Bbbk$-basis of
 $\langle A^{ij}_kT^i_j\rangle$.
Moreover, by Lemma \ref{5.3.4}, the $He_i$-module action on $\langle A^{ij}_kT^i_j\rangle$ is determined  by
$$K_iE^t_iA^{ij}_kT^i_j=\mathbf{q}^{n-i}q^{2(r_{k,j}+t)}E^t_iA^{ij}_kT^i_j$$ and
$$F_iE^{t}_iA^{ij}_kT^i_j=\gamma_{t}E^{t-1}_iA^{ij}_kT^i_j,$$
where $\gamma_t=\frac{1}{q-q^{-1}}(\mathbf{q}^{i-n}\frac{q^{-2r_{k,j}}-q^{-2(r_{k,j}+t)}}{1-q^{-2}}-\mathbf{q}^{n-i}\frac{q^{2r_{k,j}}-q^{2(r_{k,j}+t)}}{1-q^2})$
for $0\<t\<n-1$. Hence $\langle A^{ij}_kT^i_j\rangle \cong {\bf V}(n-i,r_{k,j})$.
\end{proof}
\begin{proposition}\label{5.3.7}
 Let $1\<i\<n-1$ and $1\<j\<n$. Then the sum $\langle A^{ij}_1T^i_j\rangle +\langle A^{ij}_2T^i_j\rangle +\cdots+\langle A^{ij}_nT^i_j\rangle$ is direct. Moreover, we have $HT^i_j\cong \oplus_{r=0}^{n-1}{\bf V}(n-i,r).$
\end{proposition}
\begin{proof}
One can easily check that $r_{k_1,j}\neq r_{k_2,j}$ in $\mathbb{Z}_n$ if $1\<k_1\neq k_2\<n$, where $r_{k_1,j}$ and $r_{k_2,j}$ are defined as before. By Proposition \ref{5.3.6}, one knows that $\langle A^{ij}_{k}T^i_j\rangle \cong {\bf V}(n-i,r_{k,j})$ for any $1\<k\<n$.  Then by Proposition \ref{5.2.4}, the sum $\langle A^{ij}_1T^i_j\rangle +\langle A^{ij}_2T^i_j\rangle +\cdots+\langle A^{ij}_nT^i_j\rangle$ is direct.
Furthermore, we have $$HT^i_j=\langle A^{ij}_1T^i_j\rangle \oplus \langle A^{ij}_2T^i_j\rangle \oplus \cdots\oplus\langle A^{ij}_nT^i_j\rangle $$ by comparing their dimensions. Consequently, $HT^i_j\cong \oplus_{r=0}^{n-1}{\bf V}(n-i,r).$
\end{proof}
\begin{corollary}\label{5.3.8}
 For any  $1\<i\<n-1$, $He_i$ is a  semisimple algebra.
\end{corollary}

\begin{corollary}\label{5.3.9}
Let $1\<t\<n-1$ and $r\in\mathbb{Z}$. Then ${\bf V}(t,r)$ is a projective simple $H$-module.
\end{corollary}

\begin{theorem}\label{5.3.10} The set
$\{V_l, {\bf V}(t,r), P_t, \Omega^{\pm s}V_t, M_s({t}, \eta)|1\<l\<n, 1\<t\<n-1, r\in\mathbb{Z}_n, s\>1, \eta\in \overline{\Bbbk}\}$ forms a complete set of all non-isomorphic indecomposable $H$-modules.
\end{theorem}
\begin{proof}
It follows from Corollary \ref{5.3.2}, Corollaries \ref{5.2.6}, \ref{5.3.8}, and Subsection \ref{4.1}.
\end{proof}

\section{\bf Tensor products of indecomposable $H$-modules}\selabel{5.4}
In this section, we investigate the decomposition rules for the tensor products of finite dimensional indecomposable modules over $H$.

By Subsection 2.3, $G=\langle K\rangle$ and $|G|=m$. For any $f=K^s\in G$, we denote $1_f$ by $1_s$. Then $1_s=\frac{1}{m}\sum_{r=0}^{m-1}\xi^{sr}_mK^r.$
A tedious but routine computation shows that
$$\begin{array}{rll}
\Delta(E)&=&(E\otimes 1)(T_1\otimes K^{n^2-n}+T_2\otimes 1)+K\otimes E,\\
\Delta(F)&=&(F\otimes 1)(T_3\otimes K^{n^2-1}+T_4\otimes K^{n-1})+1\otimes F,
\end{array}$$
where $T_1=\sum_{i=0}^{2n-1}1_i$, $T_2=\sum_{i=2n}^{n^2-1}1_i$, $T_3=\sum_{i=0}^{n^2-2n-1}1_i$ and $T_4=\sum_{i=n^2-2n}^{n^2-1}1_i$. Moreover, $T_1+T_2=1$ and $T_3+T_4=1$.

By the discussion in Section 3, there is an algebra epimorphism
$$\begin{array}{c}
 \phi: H\rightarrow \overline{U}_q,\
 K\mapsto  K,
 E\mapsto E,
 F\mapsto F,
 K^{-1}\mapsto K^{-1}
 \end{array}$$
and $\mathrm{Mod}\overline{U}_q$ (or equivalently, $\mathrm{Mod}He_0$) is a subcategory of $\mathrm{Mod}H$.
By taking the the opposite coalgebra structure, one gets a new Hopf algebra  $\overline{U}_q^{\rm cop}$. That is,  $\overline{U}_q^{\rm cop}$ has the same algebra structure with $\overline{U}_q$ but has the opposite comultiplication.
Thus, the above $\phi$ is also an algebra epimorphism from $H$ to $\overline{U}_q^{\rm cop}$.
A straightforward verification shows that the following two diagrams commute:
$$\begin{tikzpicture}[scale=1.0]
\path (-4.8,0) node(1) {$$} (-4.6,0) node(.) {$H$} (-1.6,0) node(.) {$\overline{U}^{\rm cop}_q$} (-2.4,0) node (.) {$$};
\path (-4.8,-2) node(1) {$$} (-4.6,-2) node(.) {$H\otimes H$} (-1.4,-2) node(.) {$\overline{U}^{\rm cop}_q\otimes \overline{U}^{\rm cop}_q$} (-2.4,-2) node (.) {$$};
\path (-4.2,0) node(a1) {} (-3.2, 0.20) node(.) {$\phi$} (-2.2, 0) node(b1) {};
\path (-4.2,-2) node(a3) {} (-3.2, -1.8) node(.) {$\phi\otimes \phi$} (-2.4, -2) node(b4) {};
\path (-4.6,-0.3) node(c1) {$$} (-1.6,-0.3) node(c2) {$$};
\path (-4.6,-1.6) node(c3) {$$} (-1.6,-1.6) node(c4) {$$};
\path (-4.4,-1) node(.) {$\Delta$} (-1.8,-1) node(.) {$\Delta$};
\draw[->] (a1) --(b1);
\draw[->] (a3) --(b4);
\draw[->] (c1) --(c3);
\draw[->] (c2) --(c4);
\
\path (0.2,0) node(1) {$$} (0.4,0) node(.) {$H$} (3.4,0) node(.) {$\overline{U}^{\rm cop}_q$} (2.6,0) node (.) {$$};
\path (0.2,-2) node(1) {$$} (0.4,-2) node(.) {$$} (3.3,-2) node(.) {$\Bbbk$} (2.6,-2) node (.) {$$};
\path (0.6,0) node(a1) {} (1.8, 0.20) node(.) {$\phi$} (3.0, 0) node(b1) {};
\path (0.8,-2) node(a3) {} (1.8, -1.8) node(.) {$$} (3.0, -2) node(b4) {};
\path (0.4,-0.2) node(c1) {$$} (3.3,-0.3) node(c2) {$$};
\path (0.4,-1.6) node(c3) {$$} (3.3,-1.8) node(c4) {$$};
\path (1.0,-1) node(.) {$\varepsilon$} (3.0,-1) node(.) {$\varepsilon$};
\draw[->] (a1) --(b1);
\draw[->] (c1) --(b4);
\draw[->] (c2) --(c4);
\end{tikzpicture}$$
Hence $\phi$ is a quasi-bialgebra epimorphism from $H$ to $\overline{U}_q^{\rm cop}$.
Note that a bialgebra (resp., Hopf algebra) is a quasi-bialgebra (resp., quasi-Hopf algebra).
So we have the following remark.
\begin{remark}
The category $\mathrm{Mod}\overline{U}^{\rm cop}_q$ is a monoidal full subcategory of $\mathrm{Mod}H$,
and consequently $r(\overline{U}^{\rm cop}_q)$ is a subring of $r(H)$. Since $\overline{U}_q$ is quasitriangular,
$M\ot N\cong N\ot M$ for any $M, N\in \mathrm{Mod}\overline{U}_q$. Thus, for any two $\overline{U}_q$-modules $M$ and $N$,
the decomposition rule of the tensor product $M\ot N$ in $\mathrm{Mod}\overline{U}_q$ is the same as that in mod$H$.
That is, $r(\overline{U}_q)$ ($=r(\overline{U}_q^{\rm cop}$)) is a subring of $r(H)$.
\end{remark}
\begin{remark}\label{s iso s}
The stable Green ring $r_{st}(\overline{U}_q)$ of the small quantum group  $\overline{U}_q$ is isomorphic to the stable Green ring $r_{st}(\widetilde{U}_q)$ of the small qusi-quantum group $\widetilde{U}_q$.
\end{remark}

The decomposition rules for the tensor products of any two indecomposable $He_0$-module has been studied in \cite[Proposition 3.1- Corollary 5.24]{ChHasSun}. Hence it remains to investigate the tensor products of the indecomposable $He_0$-modules and the simple $He_i$-modules, and the tensor products of the simple $He_i$-modules and the simple $He_j$-modules $1\<i,j\<n-1$.
We first investigate the decomposition rules for the tensor products of the simple $He_0$-modules and the simple $He_i$-modules for $1\<i\<n-1$.

\begin{lemma}\label{5.4.1}
Let $M$ be an $H$-module. If $v\in M$ and $1\<k\<m$ satisfy $Kv=\mathbf{q}^kv$, then
$1_{m-k}v=v\ { \rm and} \  1_jv=0$ for any $0\<j\<m-1$ with $j\neq m-k$.
\end{lemma}
\begin{proof}
It follows from a straightforward computation.
\end{proof}
\begin{lemma}\label{5.4.2}
Let $1\<t\<n-1$ and $r\in\mathbb{Z}$. Then
$$\begin{array}{c}
V_2\otimes {\bf V}(t,r)\cong {\bf V}(t,r)\otimes V_2\cong {\bf V}(t,r+\frac{n-1}{2})\oplus {\bf V}(t,r+\frac{n+1}{2}).
\end{array}$$
\end{lemma}
\begin{proof}
Let $M=V_2\otimes {\bf V}(t,r)$. Let $\{m_1, m_2\}$ and $\{v_1, v_2, \cdots, v_n\}$ be the standard  bases of  $V_2$ and ${\bf V}(t,r)$ respectively, as stated in Subsection $3.2$.  By Lemma \ref{5.4.1}, we have

\begin{equation*}
\begin{array}{ll}
\vspace{0.2cm}
E(m_1\otimes v_j)&=\left\{
\begin{array}{ll}
q^{-t}m_2\otimes v_j+q^{-1}m_1\otimes v_{j+1},& 1\<j\<n-1,\\
q^{-t}m_2\otimes v_n,& j=n,\\
\end{array}\right.\\
\vspace{0.2cm}
E(m_2\otimes v_j)&=\left\{
\begin{array}{ll}
qm_2\otimes v_{j+1}, & 1\<j\<n-1,\\
0,& j=n,\\
\end{array}\right.\\
\vspace{0.2cm}
F(m_1\otimes v_j)&=\left\{
\begin{array}{ll}
0,& j=1,\\
\gamma_{j-1}m_1\otimes v_{j-1},& 2\<j\<n,\\
\end{array}\right.\\
F(m_2\otimes v_j)&=\left\{
\begin{array}{ll}
\mathbf{q}^{(t-2r)n-t}m_1\otimes v_1, & j=1,\\
\mathbf{q}^{(t-2r-2(j-1))n-t}m_1\otimes v_j+\gamma_{j-1}m_2\otimes v_{j-1},& 2\<j\<n,\\
\end{array}\right.\\
\end{array}
\end{equation*}\\
$$K(m_1\otimes v_j)=\mathbf{q}^tq^{2(r+j-1)-1}m_1\otimes v_j,\ 1\<j\<n,$$ and
$$K(m_2\otimes v_j)=\mathbf{q}^tq^{2(r+j-1)+1}m_2\otimes v_j, \ 1\<j\<n.$$
Now let $\langle m_1\otimes v_1\rangle $ be the submodule of $M$ generated by $m_1\otimes v_1$. By a straightforward computation, $K(m_1\otimes v_1)=\mathbf{q}^tq^{2r-1}m_1\otimes v_1$ and $E^{n-1}(m_1\otimes v_1)=-q^tm_2\otimes v_{n-1}+qm_1\otimes v_n\neq 0$. Moreover, it is easy to see that $\{m_1\otimes v_1,E(m_1\otimes v_1),\cdots,E^{n-1}(m_1\otimes v_1)\}$ is a linearly independent set over $\Bbbk$. One can check that $F(m_1\otimes v_1)=0$ and $FE^j(m_1\otimes v_1)=\gamma'_jE^{j-1}(m_1\otimes v_1)$, where $\gamma'_j=\frac{1}{q-q^{-1}}(\mathbf{q}^{-t}\frac{q^{1-2r}-q^{1-2r-2j}}{1-q^{-2}}-\mathbf{q}^t
\frac{q^{2r-1}-q^{2r-1+2j}}{1-q^{2}})$ for $1\<j\<n-1$. Hence $\langle m_1\otimes v_1\rangle \cong {\bf V}(t,r+\frac{n-1}{2})$.
Similarly, let $u=m_1\otimes v_2-\gamma_1\mathbf{q}^{t-(t-2r)n}m_2\otimes v_1$, and let $\langle u\rangle $ be the submodule of $M$ generated by $u$. One can check that $Ku=\mathbf{q}^tq^{2r+1}u$, $E^{n-1}u=(-\gamma_1q^{-1}\mathbf{q}^{t-(t-2r)n}-q^{-t})m_2\otimes v_n\neq 0$, and $\{u,Eu,\cdots,E^{n-1}u\}$ are linearly independent over $\Bbbk$. Moreover, $Fu=0$ and $FE^ju=\gamma''_jE^{j-1}u$, where $\gamma''_j=\frac{1}{q-q^{-1}}(\mathbf{q}^{-t}\frac{q^{-2r-1}-q^{-2r-1-2j}}{1-q^{-2}}-\mathbf{q}^t
\frac{q^{2r+1}-q^{2r+1+2j}}{1-q^{2}})$, $1\<j\<n-1$. Therefore, we have $\langle u\rangle \cong {\bf V}(t,r+\frac{n+1}{2})$. It follows that $M\cong {\bf V}(t,r+\frac{n-1}{2})\oplus {\bf V}(t,r+\frac{n+1}{2})$ by comparing their dimensions.

Now let $N={\bf V}(t,r)\otimes V_2$. Then a similar argument to the above shows that $\langle v_1\otimes m_1\rangle \cong {\bf V}(t,r+\frac{n-1}{2})$ and $\langle q\gamma_1v_1\otimes m_2-v_2\otimes m_1\rangle \cong {\bf V}(t,r+\frac{n+1}{2})$. Hence $N\cong {\bf V}(t,r+\frac{n-1}{2})\oplus {\bf V}(t,r+\frac{n+1}{2})$. This completes the  proof.
\end{proof}
\begin{proposition}\label{5.4.3}
Let $1\<l\<n$, $1\<t\<n-1$ and $r\in\mathbb{Z}$.
\begin{enumerate}
	\item[(1)] If $l$ is odd, then $V_l\otimes {\bf V}(t,r)\cong {\bf V}(t,r)\otimes V_l \cong \oplus_{j=0}^{l-1}{\bf V}(t,r-\frac{l-1}{2}+j).$
	\item[(2)] If $l$ is even, then $V_l\otimes {\bf V}(t,r)\cong {\bf V}(t,r)\otimes V_l \cong \oplus_{j=0}^{l-1}{\bf V}(t,r+\frac{n-l+1}{2}+j).$
	\end{enumerate}
\end{proposition}
\begin{proof}
We only consider $V_l\otimes {\bf V}(t,r)$ since the proof for ${\bf V}(t,r)\otimes V_l$ is similar. We show it by induction on $l$. For $l=1$ and $l=2$, it follows from Lemma \ref{5.4.2} and that $V_1$ is the trivial $H$-module. Now let $2\<l<n$. If $l$ is odd, then $l-1$ and $l+1$ are even. By \cite[Proposition 3.1]{ChHasSun}, Lemma \ref{5.4.2} and the induction hypothesis, we have
 $$\begin{array}{rl}
V_2\otimes V_{l}\otimes {\bf V}(t,r)
\cong&V_{l-1}\otimes {\bf V}(t,r)\oplus V_{l+1}\otimes {\bf V}(t,r)\\
\cong&(\oplus_{j=0}^{l-2}{\bf V}(t,r+\frac{n-l+2}{2}+j))\oplus V_{l+1}\otimes {\bf V}(t,r)
\end{array}$$
and
 $$\begin{array}{rl}
V_2\otimes V_{l}\otimes {\bf V}(t,r)
\cong&\oplus_{j=0}^{l-1}V_2\otimes{\bf V}(t,r-\frac{l-1}{2}+j) \\
\cong& (\oplus_{j=0}^{l-1}{\bf V}(t,r+\frac{n-l}{2}+j))\oplus (\oplus_{j=0}^{l-1}{\bf V}(t,r+\frac{n+2-l}{2}+j)).
\end{array}$$
Thus, it follows from the Krull-Schmidt Theorem that
$$\begin{array}{c}
V_{l+1}\otimes {\bf V}(t,r)\cong \oplus_{j=0}^{l}{\bf V}(t,r+\frac{n-l}{2}+j).
\end{array}$$
If $l$ is even, then $l-1$ and $l+1$ are odd. Similarly, by \cite[Proposition 3.1]{ChHasSun} and the induction hypothesis, we have
 $$\begin{array}{rl}
V_2\otimes V_{l}\otimes {\bf V}(t,r)
\cong&V_{l-1}\otimes {\bf V}(t,r)\oplus V_{l+1}\otimes {\bf V}(t,r)\\
\cong&(\oplus_{j=0}^{l-2}{\bf V}(t,r-\frac{l-2}{2}+j))\oplus V_{l+1}\otimes {\bf V}(t,r)
\end{array}$$
and
 $$\begin{array}{rl}
V_2\otimes V_{l}\otimes {\bf V}(t,r)\cong&\oplus_{j=0}^{l-1}V_2\otimes{\bf V}(t,r+\frac{n-l+1}{2}+j) \\
\cong&(\oplus_{j=0}^{l-1}{\bf V}(t,r+\frac{2n-l}{2}+j))\oplus (\oplus_{j=0}^{l-1}{\bf V}(t,r+\frac{2n-l+2}{2}+j)).
\end{array}$$
Thus, it follows from the Krull-Schmidt Theorem that
$$\begin{array}{c}
V_{l+1}\otimes {\bf V}(t,r)\cong \oplus_{j=0}^{l}{\bf V}(t,r-\frac{l}{2}+j).
\end{array}$$
This completes the proof.
\end{proof}
	
Let $M$ and $N$ be two $H$-modules. Recall from \cite[Proposition 4.2.12]{EGNO}, that $M\otimes N$ is projective if either $M$ or $N$ is projective.

\begin{lemma}\label{L1}
Let $M$ be an $H$-module and $P$ a projective $H$-module. Assume that $S_1, S_2, \cdots, S_t$ are all non-isomorphic simple factors of $M$ with the multiplicities  $l_1, l_2, \cdots, l_t$, respectively. Then
$$M\otimes P\cong\oplus_{j=1}^tl_jS_j\otimes P \text{ and } P\otimes M\cong\oplus_{j=1}^tl_jP\otimes S_j.$$
\end{lemma}

\begin{proof}
We prove the lemma by induction on the radical length $r:={\rm rl}(M)$. If $r=1$ then $M\cong\oplus_{j=1}^tl_jS_j$.
Hence $M\otimes P\cong(\oplus_{j=1}^tl_jS_j)\otimes P\cong\oplus_{j=1}^tl_jS_j\otimes P$. Now let $r>1$. Suppose $M/{\rm rad}(M)\cong\oplus_{j=1}^tn_jS_j$. Then $(M/{\rm rad}(M))\otimes P\cong\oplus_{j=1}^tn_jS_j\otimes P$. Obviously, there is an exact sequence of $H$-modules
$$0\rightarrow{\rm rad}(M)\rightarrow M\rightarrow M/{\rm rad}(M)\rightarrow 0.$$
By tensoring with $P$, one obtains the exact sequence of $H$-modules:
$$0\rightarrow{\rm rad}(M)\otimes P\rightarrow M\otimes P\rightarrow(M/{\rm rad}(M))\otimes P\rightarrow 0.$$
 Since $(M/{\rm rad}(M))\otimes P$ is projective, the above exact sequence is split. Thus, by the induction hypothesis, we have
$$\begin{array}{rl}
M\otimes P\cong&{\rm rad}(M)\otimes P \oplus (M/{\rm rad}(M))\otimes P\\
\cong&(\oplus_{j=1}^t(l_j-n_j)S_j\otimes P)\oplus(\oplus_{j=1}^tn_jS_j\otimes P)\\
\cong&\oplus_{j=1}^tl_jS_j\otimes P.\\
\end{array}$$
Similarly, one can show that $P\otimes M\cong\oplus_{j=1}^tl_jP\otimes S_j$.
\end{proof}

Now we investigate the decomposition rules for the tensor products of the non-simple indecomposable $He_0$-modules and the simple $He_i$-modules, $1\<i\<n-1$.

\begin{proposition}\label{5.4.4}
Let $1\<l,t\<n-1$ and $r\in\mathbb{Z}$. Then
	$$P_l\otimes {\bf V}(t,r)\cong {\bf V}(t,r)\otimes P_l\cong\oplus_{j=0}^{n-1}2{\bf V}(t,j).$$
\end{proposition}
\begin{proof}
By Corollary \ref{5.3.9}, ${\bf V}(t,r)$ is projective.
By the Section 2 of \cite{Ch4}, $P_l$ has two simple factors $V_l$ and $V_{n-l}$, whose multiplicities are $2$.
Then by Propositions \ref{5.4.3} and Lemma \ref{L1}, we have
$$\begin{array}{rl}
P_l\otimes{\bf V}(t,r)\cong&2V_l\otimes{\bf V}(t,r)\oplus 2V_{n-l}\otimes{\bf V}(t,r)\\
\cong&2{\bf V}(t,r)\otimes V_l\oplus 2{\bf V}(t,r)\otimes V_{n-l}
\cong{\bf V}(t,r)\otimes P_l.\\
\end{array}$$
Furthermore, if $l$ is odd, then $n-l$ is even, and hence
$$\begin{array}{rl}
V_l\otimes {\bf V}(t,r)\cong&\oplus_{j=0}^{l-1} {\bf V}(t,r-\frac{l-1}{2}+j),\\
V_{n-l}\otimes {\bf V}(t,r)\cong&\oplus_{j=0}^{n-l-1} {\bf V}(t,r+\frac{n-(n-l)+1}{2}+j)\cong\oplus_{j=l}^{n-1} {\bf V}(t,r-\frac{l-1}{2}+j).\\	
\end{array}$$
If $l$ is even, then $n-l$ is odd, and so
$$\begin{array}{rl}
V_l\otimes {\bf V}(t,r)\cong&\oplus_{j=0}^{l-1} {\bf V}(t,r+\frac{n-l+1}{2}+j),\\
V_{n-l}\otimes {\bf V}(t,r)\cong&\oplus_{j=0}^{n-l-1} {\bf V}(t,r-\frac{n-l-1}{2}+j)\cong\oplus_{j=l}^{n-1} {\bf V}(t,r+\frac{n-l+1}{2}+j).\\
	\end{array}$$
In any case, we always have $P_l\otimes {\bf V}(t,r)\cong{\bf V}(t,r)\otimes P_l\cong\oplus_{j=0}^{n-1}2{\bf V}(t,j)$.
\end{proof}

\begin{corollary}\label{5.4.7}
Let $s\>1$, $1\<t,l\<n-1$ and  $r\in\mathbb{Z}$.
\begin{enumerate}
	\item[(1)] If both $s$ and $l$ are odd, then
$$\begin{array}{rl}
&\Omega^{\pm s}V_l\otimes {\bf V}(t,r)\cong {\bf V}(t,r)\otimes \Omega^{\pm s}V_l\\
\cong&(\oplus_{j=0}^{n-1}s{\bf V}(t,j))\oplus(\oplus_{j=l}^{n-1} {\bf V}(t,r-\frac{l-1}{2}+j)).
\end{array}$$
\item[(2)] If $s$ is odd but $l$  is even, then
$$\begin{array}{rl}
&\Omega^{\pm s}V_l\otimes {\bf V}(t,r)\cong {\bf V}(t,r)\otimes \Omega^{\pm s}V_l\\
\cong&(\oplus_{j=0}^{n-1}s{\bf V}(t,j))\oplus(\oplus_{j=l}^{n-1} {\bf V}(t,r+\frac{n-l+1}{2}+j)).
\end{array}$$
\item[(3)] If $s$ is even and $l$ is odd, then
$$\begin{array}{rl}
&\Omega^{\pm s}V_l\otimes {\bf V}(t,r)\cong {\bf V}(t,r)\otimes \Omega^{\pm s}V_l\\
\cong&(\oplus_{j=0}^{n-1}s{\bf V}(t,j))\oplus(\oplus_{j=0}^{l-1} {\bf V}(t,r-\frac{l-1}{2}+j)).
\end{array}$$
\item[(4)] If both $s$ and $l$ are even, then
$$\begin{array}{rl}
&\Omega^{\pm s}V_l\otimes {\bf V}(t,r)\cong {\bf V}(t,r)\otimes \Omega^{\pm s}V_l\\
\cong&(\oplus_{j=0}^{n-1}s{\bf V}(t,j))\oplus(\oplus_{j=0}^{l-1} {\bf V}(t,r+\frac{n-l+1}{2}+j)).
\end{array}$$
\end{enumerate}
\end{corollary}
\begin{proof}
By the Section 3 of \cite{Ch4}, $\Omega^{\pm s}V_l$ has two simple factors $V_l$ and $V_{n-l}$.
When $s$ is odd, $V_l$ and $V_{n-l}$ have the multiplicities $s$ and $s+1$, as simple factors of $\Omega^{\pm s}V_l$, respectively. When $s$ is even, $V_l$ and $V_{n-l}$ have the multiplicities $s+1$ and $s$, as simple factors of $\Omega^{\pm s}V_l$, respectively. Then the corollary follows from an argument similar to the proof of Proposition \ref{5.4.4}.
\end{proof}

\begin{corollary}\label{5.4.9}
Let $s\>1$, $1\<l,t\<n-1$, $r\in\mathbb{Z}$ and $\eta \in \overline{\Bbbk}$. Then
$$M_s(l,\eta)\otimes {\bf V}(t,r)\cong {\bf V}(t,r)\otimes M_s(l,\eta)\cong\oplus_{j=0}^{n-1}s{\bf V}(t,j).$$
\end{corollary}
\begin{proof}
By the Section 3 of \cite{Ch4}, $M_s(l,\eta)$ has two simple factors $V_l$ and $V_{n-l}$ with the same multiplicity $s$. Then the corollary follows from an argument similar to the proof of Proposition \ref{5.4.4}.
\end{proof}
Finally, we investigate the decomposition rules for the tensor products of the simple $He_i$-modules and the simple $He_j$-modules, $1\<i,j\<n-1$.

\begin{proposition}\label{5.4.10}
Let $1\<t,t'\<n-1$ with $t+t'\neq n$, and $r,r'\in\mathbb{Z}$.\\
$(1)$ If $t+t'<n$, then ${\bf V}(t,r)\otimes {\bf V}(t',r')\cong\oplus_{j=0}^{n-1}{\bf V}(t+t',j)$.\\
$(2)$ If $t+t'>n$, then ${\bf V}(t,r)\otimes {\bf V}(t',r')\cong\oplus_{j=0}^{n-1}{\bf V}(t+t'-n,j)$.
\end{proposition}
\begin{proof}
Let $M={\bf V}(t,r)\otimes {\bf V}(t',r')$. Assume that $\{m_1,m_2,\cdots,m_n\}$ and $\{v_1,v_2,\cdots,v_n\}$ are the standard bases of ${\bf V}(t,r)$
and ${\bf V}(t',r')$, respectively. Then $\{m_i\otimes v_j|1\<i,j\<n\}$ is a basis of $M$. The $H$-module actions on ${\bf V}(t,r)$ and on ${\bf V}(t',r')$ are given respectively by
 $$Km_i=\mathbf{q}^tq^{2(r+i-1)}m_i, \ Kv_i=\mathbf{q}^{t'}q^{2(r'+i-1)}v_i,\ 1\<i\<n,$$
 \begin{equation*}
\begin{array}{ll}
Em_i=\left\{
\begin{array}{ll}
m_{i+1}, & 1\<i<n,\\
0,& i=n\\
\end{array}\right.&
Ev_i=\left\{
\begin{array}{ll}
v_{i+1}, & 1\<i<n,\\
0,& i=n,\\
\end{array}\right.\\
\end{array}
\end{equation*}

 \begin{equation*}
	\begin{array}{ll}
		Fm_i=\left\{
		\begin{array}{ll}
			0, & i=1,\\
			\gamma_{i-1}m_{i-1},& 2\<i\<n\\
		\end{array}\right.&
		Fv_i=\left\{
		\begin{array}{ll}
			0, & i=1,\\
			\gamma'_{i-1}v_{i-1},& 2\<i\<n,\\
		\end{array}\right.\\
	\end{array}
\end{equation*}

where $\gamma_j$ and $\gamma'_j$ are given in Subsection 3.2, $1\<j\<n-1$. For $2\<l\<2n$ and $s\in\mathbb Z$, let $M_l={\rm span}\{m_i\otimes v_j|1\<i,j\<n,i+j=l\}$ and $T^s=\{x\in M|Kx=\mathbf{q}^{t+t'}q^{2(r+r'+s-2)}x\}$. Then $T^s={\rm span}\{m_i\otimes v_j|i+j\equiv s\ ( {\rm mod} \ n)\}$ and $M=\oplus_{l=2}^{2n}M_l=\oplus_{s=2}^{n+1} T^s$. Moreover,
\begin{equation*}
\begin{array}{ll}
{\rm dim}(M_l)=\left\{
\begin{array}{ll}
l-1, & 2\<l\<n+1,\\
2n-l+1,& n+1<l\<2n.\\
\end{array}\right.&
\end{array}
\end{equation*}
Define a linear endomorphism $f$ of $M$ by $f(x)=Fx$, $x\in M$. It is easy to see that $f(M_l)\subseteq M_{l-1}$ for all $2\<l\<2n$, where $M_1=0$. Hence ${\rm ker}f=\oplus_{l=2}^{2n}{\rm ker}f\cap M_l$ and $FM_2=0$. Obviously, dim$({\rm ker}f\cap M_2)=1$. Let $3\<l\<n+1$ and $x\in M_l$. Then $x=\sum_{i=1}^{l-1}\a_im_i\otimes v_{l-i}$ for some $\a_i\in \Bbbk$, and hence
$$\begin{array}{rl}
Fx=&\sum_{i=1}^{l-1}\a_i(F\otimes 1)(T_3\otimes K^{n^2-1}+T_4\otimes K^{n-1})(m_i\otimes v_{l-i})\\
&+\sum_{i=1}^{l-1}\a_i(1\otimes F)(m_i\otimes v_{l-i})\\
=&\sum_{i=2}^{l-1}\a_{i}\gamma_{i-1}\b_{l-i}m_{i-1}\otimes v_{l-i}+\sum_{i=1}^{l-2}\a_i \gamma'_{l-i-1} m_{i}\otimes v_{l-i-1}\\
=&\sum_{i=1}^{l-2}\a_{i+1}\gamma_{i}\b_{l-i-1}m_{i}\otimes v_{l-i-1}+\sum_{i=1}^{l-2}\a_i \gamma'_{l-i-1} m_{i}\otimes v_{l-i-1}\\
=&\sum_{i=1}^{l-2}(\a_i\gamma'_{l-i-1}+\a_{i+1}\gamma_i\b_{l-i-1})m_i\otimes v_{l-i-1},
\end{array}$$
where $0\neq\b_j\in \Bbbk$ for $1\<j\<l-2$. It follows that $x\in {\rm ker}f\Leftrightarrow$ $\a_i\gamma'_{l-i-1}+\a_{i+1}\gamma_i\b_{l-i-1}=0$ for all $1\<i\<l-2$, which implies that dim(${\rm ker}f\cap M_l)=1$.

Now let $n+1<l\<2n$ and $x\in M_l$. Then $x=\sum_{i=l-n}^{n}\a_im_i\otimes v_{l-i}$ for some $\a_i\in \Bbbk$, and hence
$$\begin{array}{rl}
Fx=&\sum_{i=l-n}^{n}\a_i(F\otimes 1)(T_3\otimes K^{n^2-1}+T_4\otimes K^{n-1})(m_i\otimes v_{l-i})\\
&+\sum_{i=l-n}^{n}\a_i(1\otimes F)(m_i\otimes v_{l-i})\\
=&\sum_{i=l-n}^{n}\a_{i}\gamma_{i-1}\b_{l-i}m_{i-1}\otimes v_{l-i}+\sum_{i=l-n}^{n}\a_i \gamma'_{l-i-1} m_{i}\otimes v_{l-i-1}\\
=&\sum_{i=l-n}^{n}\a_{i}\gamma_{i-1}\b_{l-i}m_{i-1}\otimes v_{l-i}+\sum_{i=l-n+1}^{n+1}\a_{i-1} \gamma'_{l-i} m_{i-1}\otimes v_{l-i}\\
=&\a_{l-n}\b_n\gamma_{l-n-1}m_{l-n-1}\otimes v_n+\a_{n}\gamma'_{l-n-1}m_{n}\otimes v_{l-n-1}.\\
&+\sum_{l-n+1\<i\<n}(\a_i\b_{l-i}\gamma_{i-1}+\a_{i-1}\gamma'_{l-i})m_{i-1}\otimes v_{l-i},\\
\end{array}$$ where $0\neq\b_j\in \Bbbk$ for $l-n\<j\<n.$
It follows that $x\in {\rm ker}f\Leftrightarrow \a_i=0$ for all $l-n\<i\<n$ $\Leftrightarrow x=0$. Consequently, ${\rm ker}f\cap M_l=\{0\} $ for all $n+1<l\<2n$. Thus, one gets dim$({\rm ker}f\cap T^s)=1$ for all $2\<s\<n+1$,
and ${\rm ker}f=\oplus_{s=2}^{n+1}{\rm ker}f\cap T^s$.

Let $\phi_K$ be the linear endomorphism of $M$ defined by $\phi_K(x)=Kx$, $x\in M$. Then $\mathbf{q}^{t+t'}q^{2(r+r')},\mathbf{q}^{t+t'}q^{2(r+r'+1)},\cdots,\mathbf{q}^{t+t'}q^{2(r+r'+n-1)}$ are all distinct eigenvalues of $\phi_K$. It follows from $t+t'\neq n$ that each simple submodule of $M$ is isomorphic to ${\bf V}(l, i)$ for some $1\<l\<n-1$ and $0\<i\<n-1$. Thus by Theorem \ref{5.3.10}, each indecomposable summand of $M$ is isomorphic to some ${\bf V}(l, i)$. Moreover, by dim$({\rm ker}f\cap T^s)=1$ for all $0\<s\<n-1$ and the structure of ${\bf V}(l,i)$, all the indecomposable  summands of $M$ are pairwise nonisomorphic, and  the number of the indecomposable summands of $M$ is equal to $n$.
Consequently,  if $t+t'<n$ then $M\cong \oplus_{j=0}^{n-1}{\bf V}(t+t',j)$, and if  $t+t'>n$ then
 $M\cong \oplus_{j=0}^{n-1}{\bf V}(t+t'-n,j)$.
\end{proof}

\begin{theorem}\label{5.4.11}
Let $1\<t,t'\<n-1$ and $r,r'\in\mathbb Z$. Assume  $t+t'= n$ and $r+r'\equiv u \ ({\rm mod}\ n)$ with $0\<u\<n-1$.
Let $P_0=0$.
\begin{enumerate}
	\item[(1)] If $0\<u\<\frac{n-1}{2}$, then
$${\bf V}(t,r)\otimes {\bf V}(t',r')\cong(\oplus_{j=u}^{\frac{n-1}{2}}P_{2j})\oplus V_n\oplus (\oplus_{1\<j\<u-1}P_{n-2u+2j}).$$
\item[(2)] If $\frac{n+1}{2}\<u\<n-1$, then
$$\begin{array}{c}
{\bf V}(t,r)\otimes {\bf V}(t',r')\cong(\oplus_{n+1-u\<j\<\frac{n-1}{2}}P_{2j})\oplus(\oplus_{j=0}^{n-u}P_{2u+2j-n}).\\
\end{array}$$
\end{enumerate}
\end{theorem}
\begin{proof}
Let $M={\bf V}(t,r)\otimes {\bf V}(t',r')$. We use the notations in the proof of  Proposition \ref{5.4.10}. Then $T^{s+n}=T^s$ for any $s\in\mathbb Z$, and  $M=\oplus_{l=2}^{2n}M_l=\oplus_{s=2}^{n+1} T^s$. Moreover, ${\rm dim}({\rm ker}(f)\cap T^s)=1$ for any $s$, ${\rm dim}({\rm ker}(f)\cap M_l)=1$ for $2\<l\<n+1$, ${\rm ker}(f)\cap M_l=0$ for $n+1<l\<2n$ and ${\rm ker}(f)=\oplus_{l=2}^{n+1}{\rm ker}(f)\cap M_l=\oplus_{s=2}^{n+1}{\rm ker}(f)\cap T^s$. Clearly, ${\rm ker}(f)\cap M_l={\rm ker}(f)\cap T^l$ for all for $2\<l\<n+1$. Let $g$ be the linear endomorphism of $M$ defined by $g(x)=Ex$, $x\in M$. Clearly, $g(M_l)\subseteq M_{l+1}$ for all $2\<l\<2n$, where $M_{2n+1}=0$. An argument similar to the proof of Proposition \ref{5.4.10} shows that ${\rm ker}(g)\cap M_l=0$ for all $2\<l\<n$, ${\rm dim}({\rm ker}(g)\cap M_l)=1$ for all $n+1\<l\<2n$ and ${\rm ker}(g)=\oplus_{l=n+1}^{2n}{\rm ker}(g)\cap M_l$.

By Corollary \ref{5.3.9}, $M$ is projective. Hence the summands of $M$ are all projective. Since $t+t'=n$, any simple submodule of $M$ is isomorphic to $V_l$ for some $1\<l\<n$, which is a  simple $He_0$-module. Hence any indecomposable summand of $M$ is isomorphic to some $P_l$, where $1\<l\<n$ and $P_n=V_n$. Note that $P_l$ is both a projective cover of $V_l$ and an injective envelope of $V_l$. It follows that $M$ is an injective envelope of ${\rm soc}(M)$. By the structures of the simple modules $V_l$, any simple submodule of $M$ is generated by a nonzero element of ${\rm ker}(f)\cap T^s$ for some $s$. Moreover, for any $0\neq x\in {\rm ker}(f)\cap T^s$ and $0\neq y\in {\rm ker}(f)\cap T^{s'}$ with $2\<s\neq s'\<n+1$, if the submodules $\langle x\rangle$ and
$\langle y\rangle$ are simple, then they are not isomorphic. Thus, it follows from the fact ${\rm dim}({\rm ker}(f)\cap T^s)=1$ that any two different simple submodules of $M$ are not isomorphic. Consequently, any two different indecomposable summands of $M$ are not isomorphic.

Assume $u=0$. Then $Kx=q^{2(s-2+\frac{1-n}{2})}x$ for any $x\in T^s$. Let $N$ be a simple $H$-submodule of $M$. We claim  that $N$ is not isomorphic to $V_{2k-1}$ for any $1\<k\<\frac{n-1}{2}$. Suppose on the contrary that $N\cong V_{2k-1}$ for some $1\<k\<\frac{n-1}{2}$. Then $N$ has a standard basis $\{x_1,x_2,\cdots,x_{2k-1}\}$ such that $Kx_i=q^{2(i-k)}x_i$ for all $1\<i\<2k-1$, $Fx_1=Ex_{2k-1}=0$ and $Ex_i=x_{i+1}$ for all $1\<i\<2k-2$. Hence $x_1\in{\rm ker}f\cap T^{\frac{n+5}{2}-k}$.
 Since $1\<k\<\frac{n-1}{2}$, one has $3\<\frac{n+5}{2}-k\<\frac{n+3}{2}\<n$. Hence ${\rm ker}f\cap T^{\frac{n+5}{2}-k}={\rm ker}f\cap M_{\frac{n+5}{2}-k}$, and so $x_1\in{\rm ker}f\cap M_{\frac{n+5}{2}-k}$.
Now we have $\frac{n+5}{2}-k+2k-2=\frac{n+1}{2}+k\<\frac{n+1}{2}+\frac{n-1}{2}=n$. Then by ${\rm ker}(g)\cap M_l=\{0\}$ for all $2\<l\<n$, one sees that the restriction $g^{2k-1}|_{ M_{\frac{n+5}{2}-k}}: M_{\frac{n+5}{2}-k}\rightarrow  M_{\frac{n+3}{2}+k}$ is injective. However,
$g^{2k-1}(x_1)=E^{2k-1}x_1=0$, a contradiction! Thus, we have shown the claim. Now from the discussion in the previous paragraph, one knows that ${\rm soc}(M)$ is isomorphic to a submodule of $(\oplus_{j=1}^{\frac{n-1}{2}}V_{2j})\oplus V_n$, and so $M$ is isomorphic to a submodule of $(\oplus_{j=1}^{\frac{n-1}{2}}P_{2j})\oplus V_n$.  By comparing their dimensions, one finds $M\cong(\oplus_{j=1}^{\frac{n-1}{2}}P_{2j})\oplus V_n$. In case $u=1$, a similar argument to the above shows that $M\cong(\oplus_{j=1}^{\frac{n-1}{2}}P_{2j})\oplus V_n$.

Now let $2\<u\<\frac{n-1}{2}$. Then $n\>5$ and $Kx=q^{2(u+s-2+\frac{1-n}{2})}x$ for any $x\in T^s$. Let $N$ be a simple $H$-submodule of $M$. We claim  that $N$ is not isomorphic to  $V_{2j}$ for any $1\<j\<u-1$.  Suppose on the contrary $N\cong V_{2j}$ for some $1\<j\<u-1$. Then there is a standard basis $\{x_1,x_2,\cdots,x_{2j}\}$ in $N$ such that $Kx_i=q^{2i-2j-1}x_i$ for all $1\<i\<2j$, $Fx_1=Ex_{2j}=0$ and $Ex_i=x_{i+1}$ for all $1\<i\<2j-1$. Hence $x_1\in{\rm ker}f\cap T^{n+2-u-j}$.
 Since $2\<u\<\frac{n-1}{2}$ and $1\<j\<u-1$, it follows that $4\<n+2-u-j\<n-1$. Hence ${\rm ker}f\cap T^{n+2-u-j}={\rm ker}f\cap M_{n+2-u-j}$, and so $x_1\in{\rm ker}f\cap M_{n+2-u-j}$.
Since $n+2-u-j+2j-1\<n$ and ${\rm ker}(g)\cap M_l=\{0\}$ for all $2\<l\<n$, the restriction $g^{2j}|_{ M_{n+2-u-j}}: M_{n+2-u-j}\rightarrow  M_{n+2-u+j}$ is injective. However,
$g^{2j}(x_1)=E^{2j}x_1=0$, a contradiction! Thus, we have shown the claim. Similarly, one can show that  $N$ is not isomorphic to  $V_{2j+1}$ for any $0\<j\<\frac{n-1}{2}-u$. Then a similar argument shows that
$$\begin{array}{c}
M\cong(\oplus_{j=u}^{\frac{n-1}{2}}P_{2j})\oplus (\oplus_{j=\frac{n+1}{2}-u}^{\frac{n-1}{2}}P_{2j+1})
\cong(\oplus_{j=u}^{\frac{n-1}{2}}P_{2j})\oplus (\oplus_{j=1}^uP_{n-2u+2j}).\\
\end{array}$$
This completes the proof of Part (1). The proof of Part (2) is similar.
\end{proof}
\section{\bf The representation ring of $H$}\selabel{6}

\subsection{\bf Green ring of $\overline{U}_q$}\selabel{4.2}
By the discussion in Subsection 3.1, $\mathrm{Mod}\ol{U}_q$ is a monoidal full subcategory of $\mathrm{Mod}H_n(1,q^2)$,
and hence $r(\ol{U}_q)$ is a subring of $r(H_n(1,q^2))$.
$r(H_n(1,q^2))$ was described in \cite{Ch5} for $n=2$ and in \cite{SunHasLinCh} for $n>2$, respectively.
In this subsection, we investigate the Green ring $r(\ol{U}_q)$ of $\ol{U}_q$ in case $n$ is odd.
Throughout this subsection, assume that $n>2$ is odd and $q\in\Bbbk$ is a root of unity with $|q|=n$ or $|q|=2n$.

Note that the projective class ring $r_p(\overline{U}_q)$ is commutative since $r(\ol{U}_q)$ is.
By \cite[Corollary 3.2]{ChHasSun}, the category consisting of semisimple modules and
projective modules in $\mathrm{Mod}\ol{U}_q$ is a monoidal subcategory of $\mathrm{Mod}\ol{U}_q$.
Therefore, we have the following proposition.

\begin{lemma}\label{4.2.1}
$r_p(\overline{U}_q)$ is a free $\mathbb Z$-module with a $\mathbb Z$-basis
$$\{ [V_k], [P_l]|1\<k\<n, 1\<l\<n-1\}.$$
\end{lemma}

\begin{lemma}\label{4.2.2}
Let $2\< k \< n-1$. Then
$$V_2^{\otimes k}\cong\oplus_{i=0}^{[\frac{k}{2}]}\frac{k-2i+1}{k-i+1}\binom{k}{i}V_{k+1-2i}.$$
\end{lemma}
\begin{proof}
It follows from \cite[Theorem 3.1]{Ch2} and \cite[Lemma 5.3]{ChHasLinSun} .
\end{proof}

Let $y=[V_2]$ in $r(\overline{U}_q)$.

\begin{proposition}\label{4.2.3}
The following equations hold in $r_p(\overline{U}_q)$ (or $r(\overline{U}_q)$):
\begin{enumerate}
\item[(1)] $y[V_j]=[V_{j+1}]+[V_{j-1}]$ for $2\<j\<n-1$;
\item[(2)] $y[V_n]=[P_{n-1}]$;
\item[(3)] $y[P_1]=[P_2]+2[V_n]$;
\item[(4)] $y[P_j]=[P_{j+1}]+[P_{j-1}]$ for $2\<j\<n-2$;
\item[(5)] $y[P_{n-1}]=2[V_n]+[P_{n-2}]$.
\end{enumerate}
\end{proposition}

\begin{proof}
It follows from \cite[Proposition 3.1, Theorems 3.3 and 3.5]{ChHasSun}.
\end{proof}

\begin{lemma}\label{4.2.4}
	The following equations hold in $r_p(\overline{U}_q)$.
\begin{enumerate}
\item[(1)] $[V_l]=\sum_{i=0}^{[\frac{l-1}{2}]}(-1)^{i}\binom{l-1-i}{i}y^{l-1-2i}$ for  $1\<l\<n$.
\item[(2)] $[P_l]=\sum_{i=0}^{[\frac{n-l}{2}]}(-1)^i\frac{n-l}{n-l-i}\binom{n-l-i}{i}y^{n-l-2i}[V_n]$ for $1\<l\<n-1$.
\end{enumerate}
\end{lemma}
\begin{proof}
It follows from \cite[Lemma 5.6]{ChHasLinSun}, \cite[Theorem 3.1]{Ch2} and \cite[Theorem 3.3]{ChHasSun}.
It also can be shown in a similar way to \cite[Lemma 5.6]{ChHasLinSun}.
\end{proof}

\begin{corollary}\label{4.2.5}
$r_p(\overline{U}_q)$ is generated as a ring by $y$.
\end{corollary}
\begin{proof}
It follows from Lemmas \ref{4.2.1} and \ref{4.2.4}.
\end{proof}

\begin{proposition}\label{4.2.6}
In $r(\overline{U}_q)$, we have

$$\begin{array}{c}
(\sum_{i=0}^{\frac{n-1}{2}}(-1)^i\frac{n}{n-i}\binom{n-i}{i}y^{n-2i}-2)(\sum_{i=0}^{\frac{n-1}{2}}(-1)^i\binom{n-1-i}{i}y^{n-1-2i})=0.
\end{array}$$
\end{proposition}

\begin{proof}
Since $n$ is odd, by Lemma \ref{4.2.4}(2), we have
$$\begin{array}{c}
y[P_1]=\sum_{i=0}^{\frac{n-1}{2}}(-1)^i\frac{n-1}{n-1-i}\binom{n-1-i}{i}y^{n-2i}[V_n].
\end{array}$$
By Proposition \ref{4.2.3}(3) and Lemma \ref{4.2.4}(2), we also have
$$\begin{array}{rcl}
y[P_1]&=&[P_2]+2[V_n]\\
&=&(\sum_{i=0}^{\frac{n-3}{2}}(-1)^i\frac{n-2}{n-2-i}\binom{n-2-i}{i}y^{n-2-2i}+2)[V_n]\\
&=&(\sum_{i=1}^{\frac{n-1}{2}}(-1)^{i-1}\frac{n-2}{n-1-i}\binom{n-1-i}{i-1}y^{n-2i}+2)[V_n].\\
\end{array}$$

Therefore, one obtains
$$\begin{array}{c}
(\sum_{i=0}^{\frac{n-1}{2}}(-1)^i\frac{n-1}{n-1-i}\binom{n-1-i}{i}y^{n-2i}+
\sum_{i=1}^{\frac{n-1}{2}}(-1)^i\frac{n-2}{n-1-i}\binom{n-1-i}{i-1}y^{n-2i}-2)[V_n]=0,\\
\end{array}$$
which is equivalent to $(\sum_{i=0}^{\frac{n-1}{2}}(-1)^i\frac{n}{n-i}\binom{n-i}{i}y^{n-2i}-2)[V_n]=0$.
Thus, the proposition follows from Lemma \ref{4.2.4}(1).
\end{proof}

\begin{corollary}\label{4.2.7}
$\{y^l|0\<l\<2n-2\}$ is a $\mathbb Z$-basis of $r_p(\overline{U}_q)$.
\end{corollary}

\begin{proof}
By Proposition \ref{4.2.6}, we have

$$\begin{array}{rcl}
y^{2n-1}&=&-\sum_{i=1}^{\frac{n-1}{2}}(-1)^i\binom{n-1-i}{i}y^{2n-1-2i}\\
&&-\sum_{i=1}^{\frac{n-1}{2}}(-1)^i\frac{n}{n-i}\binom{n-i}{i}y^{2n-1-2i}+2y^{n-1}\\
&&-(\sum_{i=1}^{\frac{n-1}{2}}(-1)^i\frac{n}{n-i}\binom{n-i}{i}y^{n-2i}-2)
(\sum_{i=1}^{\frac{n-1}{2}}(-1)^i\binom{n-1-i}{i}y^{n-1-2i}).\\
\end{array}$$
Then it follows from Corollary \ref{4.2.5} that $r_p(\ol{U}_q)$ is generated as a $\mathbb Z$-module by $\{y^l|0\<l\<2n-2\}$.
By Lemma \ref{4.2.1}, $r_p(\overline{U}_q)$ is a free $\mathbb Z$-module of rank $2n-1$,
and hence $\{y^l|0\<l\<2n-2\}$ forms a $\mathbb Z$-basis of $r_p(\overline{U}_q)$.
\end{proof}

\begin{theorem}\label{4.2.8}
Let $\mathbb Z[y]$ be the polynomial ring in variable $y$, and $I$ the ideal of $\mathbb Z[y]$ generated by
$$\begin{array}{c}
(\sum_{i=0}^{\frac{n-1}{2}}(-1)^i\frac{n}{n-i}\binom{n-i}{i}y^{n-2i}-2)
(\sum_{i=0}^{\frac{n-1}{2}}(-1)^i\binom{n-1-i}{i}y^{n-1-2i}).
\end{array}$$
Then $r_p(\overline{U}_q)$ is isomorphic to the quotient ring $\mathbb{Z}[y]/I$.
\end{theorem}

\begin{proof}
By Corollary \ref{4.2.5}, one can define a ring epimorphism $\phi: \mathbb{Z}[y]\rightarrow r_p(\overline{U}_q)$
by $\phi(y)=[V_2]$. By Proposition \ref{4.2.6}, $\phi(I)=0$. Hence $\phi$ induces a ring epimorphism
$\ol{\phi}: \mathbb{Z}[y]/I\rightarrow r_p(\overline{U}_q)$ such that $\phi=\ol{\phi}\circ \pi$,
where $\pi: \mathbb{Z}[y]\rightarrow\mathbb{Z}[y]/I$ is the canonical projection.
Now denote $\pi(t)$ by $\overline{t}$ for any $t\in\mathbb{Z}[y]$. Then one can easily check
that $\mathbb Z[y]/I$ is generated, as a $\mathbb Z$-module, by $\{\overline{y}^l|0\<l\<2n-2\}$,
and $\overline{\phi}(\overline{y}^l)=[V_2]^l$ for $0\<l\<2n-2$. Thus, Corollary \ref{4.2.7} implies
that $\{\overline{y}^l|0\<l\<2n-2\}$ is $\mathbb Z$-basis of $\mathbb{Z}[y]/I$.
Consequently, $\overline{\phi}$ is a $\mathbb Z$-module isomorphism. Hence, it is a ring isomorphism.
\end{proof}

Now let $z_{+}=[\O V_1]$, $z_{-}=[\O^{-1}V_1]$. Let $R$ be the subring of $r(\overline{U}_q)$ generated
by $y$, $z_+$ and $z_-$, where $y=[V_2]$ is defined as before. Then $r_p(\overline{U}_q)\subseteq R$.

\begin{proposition}\label{4.2.9}
$R$ is a free $\mathbb Z$-module with a $\mathbb Z$-basis
$$\{[V_k], [P_l], [\O^{\pm s}V_l]|1\leqslant k\leqslant n,
1\leqslant l\leqslant n-1, s\>1\}.$$
\end{proposition}

\begin{proof}
It is similar to \cite[Proposition 3.1]{SunHasLinCh}.
%
%Let $R'$ be the $\mathbb Z$-submodule of $r(\overline{U}_q)$ generated by the set
%$$\{[V_s], [P_l], [\O^{\pm m}V_l]|1\leqslant s\leqslant n,
%1\leqslant l\leqslant n-1, m\>1\}.$$
%Since the tensor product of a projective module with any module is projective,
%it follows from \cite[Corollaries 3.2, 5.3 and 5.5, Propositions 3.6, 3.7, 5.2 and 5.4]{ChHasSun}
%that $R'$ is a subring of $r(\overline{U}_q)$.
%Hence $R\subseteq R'$. It is left to show $R'\subseteq R$.
%By $r_p(\overline{U}_q)\subset R$ and \cite[Corollary 5.2]{ChHasLinSun}, it is enough to show
%$[\O^{\pm m}V_l]\in R$ for all $1\leqslant l\leqslant n-1$, $m\>1$.
%
%At first, $[\O V_1]=z_+\in R$. Let $m\>1$ and assume that $[\O^mV_1]\in R$.
%Then by \cite[Proposition 5.2]{ChHasSun}, $\O V_1\ot\O^mV_1\cong\O^{m+1}V_1\oplus P$
%for some projective $\overline{U}_q$-module $P$. Hence $[\O^{m+1}V_1]=z_+[\O^mV_1]-[P]\in R$
%by $[P]\in r_p(\overline{U}_q)\subset R$. This shows that $[\O^mV_1]\in R$ for all
%$m\>1$. Similarly, one can show that $[\O^{-m}V_1]\in R$ for all
%$m\>1$. Then for any $1\<l\<n-1$, $m\>1$, by \cite[Proposition 3.6]{ChHasSun}, we have
%$V_l\ot\O^{\pm m}V_1\cong\O^{\pm m}V_l\oplus P$ for some projective module $P$.
%By $[P], [V_l]\in r_p(\overline{U}_q) \subset R$, it follows that
%$[\O^{\pm m}V_l]=[V_l][\O^{\pm m}V_1]-[P]\in R$.
%This completes the proof.
\end{proof}

%\begin{lemma}\label{3.10}
%Let $l, m\in\mathbb{Z}$ with $1\<l\<\frac{m-1}{2}$. Then for any $s\>0$,
%$$\begin{array}{c}
%\sum_{i=0}^{s}(-1)^i\frac{m-2l+2i}{m-2l+i}\binom{m-2l+i}{i}=(-1)^s\binom{m-2l+s}{s}.\\
%\end{array}$$
%In particular, $\sum_{i=0}^{l-1}(-1)^i\frac{m-2l+2i}{m-2l+i}\binom{m-2l+i}{i}=(-1)^{l-1}\binom{m-l-1}{l-1}$.
%\end{lemma}
%
%\begin{proof}
%It can be shown by induction $s$.
%\end{proof}
Similarly to \cite[Lemma 3.3]{SunHasLinCh}, one can show the following lemma.

\begin{lemma}\label{4.2.10}
The following relations are satisfied in $r(\overline{U}_q)$.
\begin{enumerate}
\item[(1)] $\sum_{i=0}^{\frac{n-1}{2}}[P_{2i+1}]=[V_n]^2$.
\item[(2)] $\sum_{i=1}^{\frac{n-1}{2}}[P_{2i+1}]
=\sum_{i=1}^{\frac{n-1}{2}}(-1)^{i-1}\binom{n-i-2}{i-1}y^{n-1-2i}[V_n]$.
\item[(3)] $\sum_{i=1}^{\frac{n-1}{2}}[P_{2i}]
=\sum_{i=1}^{\frac{n-1}{2}}(-1)^{i-1}\binom{n-i-1}{i-1}y^{n-2i}[V_n]$.
\end{enumerate}
\end{lemma}

%\begin{proof}
%(1) Since $n$ is odd, one gets from \cite[Proposition 3.1(2)]{ChHasSun} that
%$$\begin{array}{c}
%V_n\ot V_n\cong\oplus_{i=\frac{n-1}{2}}^{n-1}P_{2n-1-2i}
%=\oplus_{i=0}^{\frac{n-1}{2}}P_{2i+1}.\\
%\end{array}$$
%Therefore, $\sum_{i=0}^{\frac{n-1}{2}}[P_{2i+1}]=[V_n]^2$.
%
%(2) By (1) and Lemma \ref{3.4}, we have
%$$\begin{array}{rcl}
%\sum_{i=1}^{\frac{n-1}{2}}[P_{2i+1}]&=&[V_n]^2-[P_1]\\
%&=&\sum_{i=0}^{\frac{n-1}{2}}(-1)^i\binom{n-1-i}{i}y^{n-1-2i}[V_n]\\
%&&-\sum_{i=0}^{\frac{n-1}{2}}(-1)^i\frac{n-1}{n-1-i}\binom{n-1-i}{i}y^{n-1-2i}[V_n]\\
%&=&\sum_{i=1}^{\frac{n-1}{2}}(-1)^{i-1}(\frac{n-1}{n-1-i}\binom{n-1-i}{i}-\binom{n-1-i}{i})y^{n-1-2i}[V_n]\\
%&=&\sum_{i=1}^{\frac{n-1}{2}}(-1)^{i-1}\binom{n-i-2}{i-1}y^{n-1-2i}[V_n].\\
%\end{array}$$
%
%(3) By Lemma \ref{3.4} and \cite[Lemma 3.2]{SunHasLinCh}, we have
%$$\begin{array}{rcl}
%\sum_{i=1}^{\frac{n-1}{2}}[P_{2i}]
%&=&\sum_{i=1}^{\frac{n-1}{2}}\sum_{j=0}^{\frac{n-2i-1}{2}}(-1)^j\frac{n-2i}{n-2i-j}
%\binom{n-2i-j}{j}y^{n-2i-2j}[V_n]\\
%&=&(\sum_{l=1}^{\frac{n-1}{2}}\sum_{j=0}^{l-1}(-1)^j\frac{n-2l+2j}{n-2l+j}
%\binom{n-2l+j}{j}y^{n-2l})[V_n]\\
%&=&(\sum_{l=1}^{\frac{n-1}{2}}(-1)^{l-1}\binom{n-l-1}{l-1}y^{n-2l})[V_n].\\
%\end{array}$$
%\end{proof}

Similarly to \cite[Lemma 3.4]{SunHasLinCh}, one can check the following lemma.

\begin{lemma}\label{4.2.11}
The following relations are satisfied in $r(\overline{U}_q)$.
\begin{enumerate}
\item[(1)] $z_+z_-=1+(2y+4\sum_{i=1}^{\frac{n-1}{2}}(-1)^{i-1}\binom{n-i-2}{i-1}y^{n-1-2i})[V_n]$.
\item[(2)] $z_{+}[V_n]=z_{-}[V_n]=(1+2\sum_{i=1}^{\frac{n-1}{2}}(-1)^{i-1}\binom{n-i-1}{i-1}y^{n-2i})[V_n]$.
\end{enumerate}
\end{lemma}

Similarly to \cite[Proposition 3.5]{SunHasLinCh}, one can show the following proposition.

\begin{proposition}\label{4.2.12}
The following set is also a $\mathbb Z$-basis of $R$:
$$\{y^j, y^lz_+^k, y^lz_-^k| 0\<j\<2n-2, 0\<l\<n-2, k\>1\}.$$
\end{proposition}

Now let $w_{k,\eta}=[M_k(1,\eta)]$ in $r(\overline{U}_q)$ for any $k\>1$ and $\eta\in\ol{\Bbbk}$.
Then similarly to \cite[Lemma 3.6]{SunHasLinCh}, one can check the following lemma.

\begin{lemma}\label{4.2.13}
Let $k, s\>1$ and $\eta, \a\in\ol{\Bbbk}$. Then we have the following relations in $r(\overline{U}_q)$.
\begin{enumerate}
\item[(1)] $w_{k,\eta}[V_n]=k(1+\sum_{i=1}^{\frac{n-1}{2}}(-1)^{i-1}\binom{n-1-i}{i-1}y^{n-2i})[V_n]$.
\item[(2)] $z_+w_{k,\eta}=(\sum_{i=1}^{\frac{n-1}{2}}(-1)^{i-1}\binom{n-1-i}{i-1}y^{n-2i})w_{k,\eta}+k[V_n]^2$.
\item[(3)] $z_-w_{k,\eta}=(\sum_{i=1}^{\frac{n-1}{2}}(-1)^{i-1}\binom{n-i-1}{i-1}y^{n-2i})w_{k,\eta}$\\
\mbox{\hspace{1.5cm}}$+k(y+\sum_{i=1}^{\frac{n-1}{2}}(-1)^{i-1}\binom{n-i-2}{i-1}y^{n-1-2i})[V_n]$.
\item[(4)] If $\eta\neq\a$, then $w_{k,\eta}w_{s,\a}=ks[V_n]^2$.
\item[(5)] If $k\<s$, then
$$\begin{array}{c}
w_{k,\eta}w_{s,\eta}=w_{k,\eta}(1+\sum_{i=1}^{\frac{n-1}{2}}(-1)^{i-1}\binom{n-i-1}{i-1}y^{n-2i})
+(s-1)k[V_n]^2.\\
\end{array}$$
\end{enumerate}
\end{lemma}

Similarly to \cite[Proposition 3.7]{SunHasLinCh}, one can prove the following proposition.

\begin{proposition}\label{4.2.14}
$r(\overline{U}_q)$ is generated, as a ring, by $$\{y, z_+, z_-, w_{k,\eta}|k\>1, \eta\in\ol{\Bbbk}\}.$$
\end{proposition}

Similarly to \cite[Proposition 3.8]{SunHasLinCh}, one can show the following proposition.

\begin{proposition}\label{4.2.15}
The following set forms a $\mathbb Z$-basis of $r(\overline{U}_q)$:
$$\left\{y^j, y^lz_+^k, y^lz_-^k, y^lw_{k,\eta}
\left|0\<j\<2n-2,
0\<l\<n-2, k\>1, \eta\in\ol{\Bbbk}
\right\}\right..$$
\end{proposition}

Let $\mathbb{Z}[y,z_+,z_-]$ be the polynomial ring in variables $ y, z_+, z_-$.
Define the polynomials $f_1(y)$, $f_2(y)$, $f_3(y)$ and $f_4(y)$ in $\mathbb{Z}[y]\subset\mathbb{Z}[y,z_+,z_-]$ by
$$\begin{array}{l}
f_1(y)=\sum_{i=0}^{\frac{n-1}{2}}(-1)^i\binom{n-1-i}{i}y^{n-1-2i},\\
f_2(y)=\sum_{i=0}^{\frac{n-1}{2}}(-1)^i\frac{n}{n-i}\binom{n-i}{i}y^{n-2i}-2,\\
f_3(y)=\sum_{i=1}^{\frac{n-1}{2}}(-1)^{i-1}\binom{n-i-2}{i-1}y^{n-1-2i},\\
f_4(y)=\sum_{i=1}^{\frac{n-1}{2}}(-1)^{i-1}\binom{n-i-1}{i-1}y^{n-2i}.\\
\end{array}$$
Let $I$ be the ideal of $\mathbb{Z}[y,z_+,z_-]$ generated by the following elements:
$$\begin{array}{c}
  f_1(y)f_2(y),\  z_+z_--1-f_1(y)(2y+4f_3(y)),\\
 f_1(y)(z_+-1-2f_4(y)),\ f_1(y)(z_+-z_-).\\
 \end{array}$$
Similar to \cite[Proposition 3.9]{SunHasLinCh}, one can show the following proposition.

\begin{proposition}\label{4.2.16}
$R$ is isomorphic to the quotient ring $\mathbb{Z}[y,z_+,z_-]/I$.
\end{proposition}

Let $X=\{ y, z_+, z_-, w_{k,\eta}|k\>1, \eta\in\ol{\Bbbk}\}$, and let $\mathbb{Z}[X]$
be the corresponding polynomial ring.
Then $\mathbb{Z}[y,z_+,z_-]$ is a subring of $\mathbb{Z}[X]$. Let $U$ be the following subset of $\mathbb{Z}[X]$:
$$U=\left\{\left.\begin{array}{l}
  f_1(y)f_2(y),\\
z_+z_--1-f_1(y)(2y+4f_3(y)),\\
f_1(y)(z_+-1-2f_4(y)),\ f_1(y)(z_+-z_-),\\
f_1(y)(w_{k,\eta}-k-kf_4(y)),\\
(z_+-f_4(y))w_{k,\eta}-kf_1(y)^2,\\
(z_--f_4(y))w_{k,\eta}-kf_1(y)(y+f_3(y)),\\
w_{k,\eta}w_{s,\a}-ksf_1(y)^2,\\
w_{k,\eta}(w_{t,\eta}-1-f_4(y))-(t-1)kf_1(y)^2\\
 \end{array}
\right|\begin{array}{l}
k, s, t\>1\\
\mbox{with } k\<t,\\
\eta, \a\in\ol{\Bbbk}\\
\mbox{with } \eta\neq\a\\
\end{array}\right\},$$
where $f_1(y)$, $f_2(y)$, $f_3(y)$, $f_4(y)\in\mathbb{Z}[y]\subset\mathbb{Z}[y,z_+,z_-]\subset\mathbb{Z}[X]$
are given as before.
Let $I_0=(U)$, the ideal of $\mathbb{Z}[X]$ generated by $U$.
Now similar to \cite[Theorem 3.10]{SunHasLinCh}, one can show the following theorem.

\begin{theorem}\label{4.2.17}
When $|q|=n$ or $|q|=2n$ for some odd $n>2$, the Green ring $r(\overline{U}_q)$ is isomorphic to $\mathbb{Z}[X]/I_0$.
\end{theorem}

\subsection{\bf Green ring of $\widetilde{U}_q$}\selabel{5.5}
In this subsection, we investigate the projective class ring $r_p(\widetilde{U}_q)$ and the Green ring $r(\widetilde{U}_q)$ of the small quasi-quantum group $\widetilde{U}_q$. First of all, we have $r(\overline{U}_q)\subseteq r(\widetilde{U}_q)$ and $r_p(\overline{U}_q)\subseteq r_p(\widetilde{U}_q)$.
It follows from \seref{5.4} that both $r_p(\widetilde{U}_q)$ and $r(\widetilde{U}_q)$ are commutative.

By \cite[Corollary 3.2]{ChHasSun}, Corollary \ref{5.2.6}, Theorem \ref{5.3.10} and  the discussion in \seref{5.4}, it follows that the subcategory consisting of semisimple $H$-modules and projective $H$-modules is a monoidal subcategory of $\mathrm{Mod}H$. Hence we have the following lemma.

\begin{lemma}\label{5.5.1}
$r_p(\widetilde{U}_q)$ is a free $\mathbb Z$-module with a $\mathbb Z$-basis
$$\{ [V_l], [P_t], [{\bf V}(t,r)]|1\<l\<n, 1\<t\<n-1, r\in\mathbb{Z}_n\}.$$
\end{lemma}

For any $1\<t\<n-1$, let $x_t=[{\bf V}(t,0)]$, $\varepsilon_t=[{\bf V}(t,1)]$ and $y=[V_2]$ in $r(\widetilde{U}_q)$.
\begin{lemma}\label{5.5.2}
Let $1\<t\<n-1$ and $r\>2$. Then
$$\begin{array}{rl}
[{\bf V}(t,r)]=&\sum_{i=0}^{[\frac{r-1}{2}]}(-1)^{i}\binom{r-1-i}{i}(y^2-2)^{r-1-2i}\varepsilon_t\\
&- \sum_{i=0}^{[\frac{r-2}{2}]}(-1)^{i}\binom{r-2-i}{i}(y^2-2)^{r-2-2i}x_t.\\
\end{array}$$
\end{lemma}
\begin{proof}
We work by induction on $r$. Let $k\in\mathbb Z$. By Proposition \ref{5.4.3}, one gets $$V_3\otimes {\bf V}(t,k)\cong {\bf V}(t,k-1)\oplus {\bf V}(t,k)\oplus {\bf V}(t,k+1).$$ Then $[{\bf V}(t,k+1)]=([V_3]-1)[{\bf V}(t,k)]-[{\bf V}(t,k-1)].$ By Proposition \ref{4.2.3}, we have $[V_3]=y^2-1$. It follows that
$$[{\bf V}(t,k+1)]=(y^2-2)[{\bf V}(t,k)]-[{\bf V}(t,k-1)].$$ Putting $k=1$, one gets the lemma for $r=2$. Putting $k=2$, one gets the lemma for $r=3$. Now assume $r\>3$. Then by the above equation  and the induction hypothesis, we have
$$\begin{array}{rl}
[{\bf V}(t,r+1)]=&(y^2-2)[{\bf V}(t,r)]-[{\bf V}(t,r-1)]\\
=&(y^2-2)(\sum_{i=0}^{[\frac{r-1}{2}]}(-1)^{i}\binom{r-1-i}{i}(y^2-2)^{r-1-2i}\varepsilon_t\\ &-\sum_{i=0}^{[\frac{r-2}{2}]}(-1)^{i}\binom{r-2-i}{i}(y^2-2)^{r-2-2i}x_t)\\
&-\sum_{i=0}^{[\frac{r-2}{2}]}(-1)^{i}\binom{r-2-i}{i}(y^2-2)^{r-2-2i}\varepsilon_t\\ &+\sum_{i=0}^{[\frac{r-3}{2}]}(-1)^{i}\binom{r-3-i}{i}(y^2-2)^{r-3-2i}x_t\\
=&\sum_{i=0}^{[\frac{r}{2}]}(-1)^{i}\binom{r-i}{i}(y^2-2)^{r-2i}\varepsilon_t\\ &-\sum_{i=0}^{[\frac{r-1}{2}]}(-1)^{i}\binom{r-1-i}{i}
 (y^2-2)^{r-1-2i}x_t.
\end{array}$$
This completes the proof.
\end{proof}

\begin{corollary}\label{5.5.3}
 $r_p(\widetilde{U}_q)$ is generated as a ring by $\{y,x_t,\varepsilon_t|1\<t\<n-1\}$.
\end{corollary}
\begin{proof}
Follows from Corollary \ref{4.2.5} and Lemmas \ref{5.5.1} and \ref{5.5.2}.
\end{proof}

\begin{proposition}\label{5.5.6}
Let $1\<t\<n-1$ and $r\>1$.
\begin{enumerate}
\item[(1)] If $r$ is odd, then
$$\begin{array}{l}
\vspace{0.1cm}
\sum_{i=0}^{\frac{r-1}{2}}(-1)^i\frac{r}{r-i}\binom{r-i}{i}y^{r-2i}x_t=[{\bf V}(t,\frac{n-r}{2})]+ [{\bf V}(t,\frac{n+r}{2})],\\
\sum_{i=0}^{\frac{r-1}{2}}(-1)^i\frac{r}{r-i}\binom{r-i}{i}y^{r-2i}\varepsilon_t=[{\bf V}(t,1+\frac{n-r}{2})]+ [{\bf V}(t,1+\frac{n+r}{2})].\\
\end{array}$$
\item[(2)] If $r$ is even, then
$$\begin{array}{l}
\vspace{0.1cm}
\sum_{i=0}^{\frac{r}{2}}(-1)^i\frac{r}{r-i}\binom{r-i}{i}y^{r-2i}x_t=[{\bf V}(t,-\frac{r}{2})]+ [{\bf V}(t,\frac{r}{2})],\\
\sum_{i=0}^{\frac{r}{2}}(-1)^i\frac{r}{r-i}\binom{r-i}{i}y^{r-2i}\varepsilon_t=[{\bf V}(t,1-\frac{r}{2})]+ [{\bf V}(t,1+\frac{r}{2})].\\
\end{array}$$
\end{enumerate}
 \end{proposition}
\begin{proof}
We only prove the first equations of $(1)$ and $(2)$ since the proofs are similar for the second equations. We work  by induction on $r$. For $r=1$, it follows from Lemma \ref{5.4.2}. For $r=2$, it follows from the proof of Lemma \ref{5.5.2}. Now let $r\>2$. If $r$ is even, then $r-1$ and $r+1$ are odd. In this case, by the induction hypothesis and Lemma \ref{5.4.2}, we have
$$\begin{array}{rl}
\vspace{0.1cm}
&\sum_{i=0}^{\frac{r}{2}}(-1)^i\frac{r}{r-i}\binom{r-i}{i}y^{r+1-2i}x_t=y[{\bf V}(t, -\frac{r}{2})]+y[{\bf V}(t, \frac{r}{2})]\\
=&[{\bf V}(t,\frac{n-1-r}{2})]+ [{\bf V}(t,\frac{n+1-r}{2})]+[{\bf V}(t,\frac{n-1+r}{2})]+[{\bf V}(t,\frac{n+1+r}{2})]\\
\end{array}$$
and
$$\begin{array}{c}
\sum_{i=0}^{\frac{r-2}{2}}(-1)^i\frac{r-1}{r-1-i}\binom{r-1-i}{i}y^{r-1-2i}x_t=[{\bf V}(t,\frac{n-(r-1)}{2})]+ [{\bf V}(t,\frac{n+(r-1)}{2})].\\
\end{array}$$
Then we have
$$\begin{array}{rl}
\vspace{0.1cm}
&[{\bf V}(t,\frac{n-(r+1)}{2})]+ [{\bf V}(t,\frac{n+(r+1)}{2})]\\
\vspace{0.1cm}
=&\sum_{i=0}^{\frac{r}{2}}(-1)^i\frac{r}{r-i}\binom{r-i}{i}y^{r+1-2i}x_t-\sum_{i=0}^{\frac{r-2}{2}}
(-1)^i\frac{r-1}{r-1-i}\binom{r-1-i}{i}y^{r-1-2i}x_t \\
=&\sum_{i=0}^{\frac{r}{2}}(-1)^i\frac{r+1}{r+1-i}\binom{r+1-i}{i}y^{r+1-2i}x_t.\\
\end{array}$$

Similarly, if $r$ is odd, then $r-1$ and $r+1$ are even. Hence
$$\begin{array}{rl}
\vspace{0.1cm}
&\sum_{i=0}^{\frac{r-1}{2}}(-1)^i\frac{r}{r-i}\binom{r-i}{i}y^{r+1-2i}x_t=y[{\bf V}(t,\frac{n-r}{2})]+y[{\bf V}(t,\frac{n+r}{2})]\\
=& [{\bf V}(t, -\frac{r+1}{2})]+ [{\bf V}(t, -\frac{r-1}{2})]+[{\bf V}(t,\frac{r-1}{2})]+[{\bf V}(t,\frac{r+1}{2})]
\end{array}$$
and
$$\begin{array}{c}
\sum_{i=0}^{\frac{r-1}{2}}(-1)^i\frac{r-1}{r-1-i}\binom{r-1-i}{i}y^{r-1-2i}x_t=[{\bf V}(t, -\frac{r-1}{2})]+ [{\bf V}(t, \frac{r-1}{2})].\\
\end{array}$$
Thus, we have
$$\begin{array}{rl}
\vspace{0.1cm}
&[{\bf V}(t, -\frac{r+1}{2})]+ [{\bf V}(t, \frac{r+1}{2})]\\
\vspace{0.1cm}
=&\sum_{i=0}^{\frac{r-1}{2}}(-1)^i\frac{r}{r-i}\binom{r-i}{i}y^{r+1-2i}x_t-\sum_{i=0}^{\frac{r-1}{2}}
(-1)^i\frac{r-1}{r-1-i}\binom{r-1-i}{i}y^{r-1-2i}x_t\\
=&\sum_{i=0}^{\frac{r+1}{2}}(-1)^i\frac{r+1}{r+1-i}\binom{r+1-i}{i}y^{r+1-2i}x_t.\\
\end{array}$$
This completes the proof.
\end{proof}

\begin{corollary}\label{5.5.8}
Let $1\<t\<n-1$. Then in $r_p(\widetilde{U}_q)$, we have
$$\begin{array}{c}
(\sum_{k=1}^{\frac{n-1}{2}}\sum_{j=0}^{[\frac{n-1-2k}{4}]}(-1)^j\frac{k+2j}{k+j}\binom{k+j}{j}y^{k}
+\sum_{1\<j\<[\frac{n-1}{4}]}(-1)^j2+1)(x_t-\varepsilon_t)=0.
\end{array}$$
\end{corollary}
\begin{proof}
 By Proposition \ref{5.5.6}, one can check that
 $$\begin{array}{c}
 (\sum_{r=1}^{\frac{n-1}{2}}\sum_{i=0}^{[\frac{r}{2}]}(-1)^i\frac{r}{r-i}\binom{r-i}{i}y^{r-2i}+1)x_t=\sum_{j=0}^{n-1}[{\bf V}(t,j)]\\
 \end{array}$$
and
$$\begin{array}{c}
(\sum_{r=1}^{\frac{n-1}{2}}\sum_{i=0}^{[\frac{r}{2}]}(-1)^i\frac{r}{r-i}\binom{r-i}{i}y^{r-2i}+1)\varepsilon_t=\sum_{j=0}^{n-1}[{\bf V}(t,j)].\\
\end{array}$$
By a straightforward computation, one obtains
$$\begin{array}{rl}
\vspace{0.1cm}
&\sum_{r=1}^{\frac{n-1}{2}}\sum_{i=0}^{[\frac{r}{2}]}(-1)^i\frac{r}{r-i}\binom{r-i}{i}y^{r-2i}\\
=&\sum_{1\<j\<[\frac{n-1}{4}]}(-1)^j2+\sum_{k=1}^{\frac{n-1}{2}}\sum_{j=0}^{[\frac{n-1-2k}{4}]}(-1)^j\frac{k+2j}{k+j}\binom{k+j}{j}y^{k}.
\end{array}$$
Thus, the corollary follows.
\end{proof}

\begin{corollary}\label{5.5.5}
Let $1\<t,t'\<n-1$. Then the following hold in $r_p(\widetilde{U}_q)$.\\
\begin{enumerate}
\item[(1)] If $t+t'<n$, then
$$\begin{array}{rl}
&x_tx_{t'}=x_{t}\varepsilon_{t'}=\varepsilon_t \varepsilon_{t'}\\
=&\sum_{k=1}^{\frac{n-1}{2}}\sum_{j=0}^{[\frac{n-1-2k}{4}]}(-1)^j\frac{k+2j}{k+j}\binom{k+j}{j}y^{k}
+\sum_{1\<j\<[\frac{n-1}{4}]}(-1)^j2+1)x_{t+t'}.
\end{array}$$
\item[(2)] If $t+t'>n$, then
$$\begin{array}{rl}
&x_tx_{t'}=x_{t}\varepsilon_{t'}=\varepsilon_t \varepsilon_{t'}\\
=&(\sum_{k=1}^{\frac{n-1}{2}}\sum_{j=0}^{[\frac{n-1-2k}{4}]}(-1)^j\frac{k+2j}{k+j}\binom{k+j}{j}y^{k}
+\sum_{1\<j\<[\frac{n-1}{4}]}(-1)^j2+1)x_{t+t'-n}.
\end{array}$$
\item[(3)] If $t+t'=n$, then
$$\begin{array}{rl}
x_tx_{t'}=x_{t}\varepsilon_{t'}
=&(1+\sum_{i=1}^{\frac{n-1}{2}}(-1)^{i-1}\binom{n-1-i}{i-1}y^{n-2i})\\
&\times(\sum_{i=0}^{\frac{n-1}{2}}(-1)^{i}\binom{n-1-i}{i}y^{n-1-2i}).\\
\end{array}$$
and
$$\begin{array}{rl}
\varepsilon_{t}\varepsilon_{t'}
=&(\sum_{2\<i\<\frac{n-1}{2}}(-1)^{i}\binom{n-2-i}{i-2}y^{n-2i}+y^2-1)\\
&\times(\sum_{i=0}^{\frac{n-1}{2}}(-1)^{i}\binom{n-1-i}{i}y^{n-1-2i}).\\
\end{array}$$
\end{enumerate}
\end{corollary}
\begin{proof}
Part (1) and Part (2) follow from Proposition \ref{5.4.10} and the proof of Corollary \ref{5.5.8}, Part (3) follows from Theorem \ref{5.4.11}, Lemmas \ref{4.2.4} and \ref{4.2.10}.
\end{proof}

\begin{proposition}\label{5.5.7}
Let $1\<t\<n-1$. Then in $r_p(\widetilde{U}_q)$, we have
$$\begin{array}{c}
(\sum_{i=0}^{[\frac{n+1}{4}]}(-1)^i\frac{n+1}{n+1-2i}\binom{\frac{n+1}{2}-i}{i}y^{\frac{n+1}{2}-2i}-
\sum_{i=0}^{[\frac{n-1}{4}]}(-1)^i\frac{n-1}{n-1-2i}\binom{\frac{n-1}{2}-i}{i}y^{\frac{n-1}{2}-2i})x_t=0.
\end{array}$$
\end{proposition}
\begin{proof}
If $\frac{n-1}{2}$ is odd, then $\frac{n+1}{2}$ is even. In this case, by Proposition \ref{5.5.6}, we have
$$\begin{array}{rl}
\sum_{i=0}^{\frac{n-3}{4}}(-1)^i\frac{n-1}{n-1-2i}\binom{\frac{n-1}{2}-i}{i}y^{\frac{n-1}{2}-2i}x_t
&=[{\bf V}(t, \frac{n+1}{4})]+[{\bf V}(t,\frac{3n-1}{4})]\\
&=[{\bf V}(t, \frac{n+1}{4})]+[{\bf V}(t, -\frac{n+1}{4})].
\end{array}$$
and
$$\begin{array}{rl}
\sum_{i=0}^{\frac{n+1}{4}}(-1)^i\frac{n+1}{n+1-2i}\binom{\frac{n+1}{2}-i}{i}y^{\frac{n+1}{2}-2i}x_t
=[{\bf V}(t, -\frac{n+1}{4})]+[{\bf V}(t,\frac{n+1}{4})].
\end{array}$$
It follows that the equation in the proposition holds in case $\frac{n-1}{2}$ is odd.

Similarly, if $\frac{n-1}{2}$ is even, then  $\frac{n+1}{2}$ is odd. By Proposition \ref{5.5.6}, we have
$$\begin{array}{c}
\sum_{i=0}^{\frac{n-1}{4}}(-1)^i\frac{n-1}{n-1-2i}\binom{\frac{n-1}{2}-i}{i}y^{\frac{n-1}{2}-2i}x_t=[{\bf V}(t,-\frac{n-1}{4})]+ [{\bf V}(t,\frac{n-1}{4})]\\
\end{array}$$ and
 $$\begin{array}{rl}
\sum_{i=0}^{\frac{n-1}{4}}(-1)^i\frac{n+1}{n+1-2i}\binom{\frac{n+1}{2}-i}{i}y^{\frac{n+1}{2}-2i}x_t
&=[{\bf V}(t,\frac{n-1}{4})]+ [{\bf V}(t,\frac{3n+1}{4})]\\
&=[{\bf V}(t,\frac{n-1}{4})]+ [{\bf V}(t,-\frac{n-1}{4})].
\end{array}$$
Thus, the equation in the proposition holds in case $\frac{n-1}{2}$ is even.
\end{proof}

\begin{corollary}\label{5.5.9}
The following set forms a $\mathbb Z$-basis of $r_p(\widetilde{U}_q)$:
$$\begin{array}{c}
\left\{y^i,y^jx_t,y^l\varepsilon_t\left|0\<i\<2n-2,0\<j\<\frac{n-1}{2},0\<l\<\frac{n-3}{2},1\<t\<n-1\right\}\right..\end{array}$$
\end{corollary}
\begin{proof}
By Corollary \ref{4.2.7}, Lemma \ref{5.5.2}, Corollary  \ref{5.5.8} and Proposition \ref{5.5.7},  one knows that $r_p(\widetilde{U}_q)$ is generated as a $\mathbb Z$-module by $$\begin{array}{c}
\left\{y^i,y^jx_t,y^l\varepsilon_t\left|0\<i\<2n-2,0\<j\<\frac{n-1}{2},0\<l\<\frac{n-3}{2},1\<t\<n-1\right\}\right..\end{array}$$
 By Proposition \ref{5.5.1}, $r_p(\widetilde{U}_q)$ is a free $\mathbb Z$-module of rank $n^2+n-1$, and so the above set  is a $\mathbb Z$-basis of $r_p(\widetilde{U}_q)$.
\end{proof}
Let $X'=\{y,x_t,\varepsilon_t|1\<t\<n-1\}$, and let $\mathbb Z[X']$ be the corresponding polynomial ring. Define  polynomials  $g_1(y)$, $g_2(y)$, $g_3(y)$ and  $g_4(y)$ in  $\mathbb Z[y]\subset\mathbb Z[X']$ by
$$\begin{array}{l}
 \vspace{0.1cm}
 g_1(y)=\sum_{i=0}^{[\frac{n+1}{4}]}(-1)^i\frac{n+1}{n+1-2i}\binom{\frac{n+1}{2}-i}{i}y^{\frac{n+1}{2}-2i},\\
  \vspace{0.1cm}
 g_2(y)=\sum_{i=0}^{[\frac{n-1}{4}]}(-1)^i\frac{n-1}{n-1-2i}\binom{\frac{n-1}{2}-i}{i}y^{\frac{n-1}{2}-2i},\\
 \vspace{0.1cm}
 g_3(y)=\sum_{2\<i\<\frac{n-1}{2}}(-1)^{i}\binom{n-2-i}{i-2}y^{n-2i}+y^2-1,\\
 g_4(y)=\sum_{k=1}^{\frac{n-1}{2}}\sum_{i=0}^{[\frac{n-1-2k}{4}]}(-1)^i\frac{k+2i}{k+i}\binom{k+i}{i}y^{k}
 +\sum_{1\<i\<[\frac{n-1}{4}]}(-1)^i2+1.
\end{array}$$
Take a subset $I$ of $\mathbb Z[X']$ by
 $$I=\left\{\left.\begin{array}{l}
  f_1(y)f_2(y),\ (g_1(y)-g_2(y))x_t,\\
g_4(y)(x_t-\varepsilon_t),\\
\ x_tx_{t'}-x_t\varepsilon_{t'},
 \ x_{t}x_{t_1}-\varepsilon_{t}\varepsilon_{t_1},\\
x_{t}x_{t_2}-g_4(y)x_{t+t_2},\\
x_{t}x_{t_3}-g_4(y)x_{t+t_3-n},\\
 x_{t}x_{n-t}-(f_4(y)+1)f_1(y), \\
 \varepsilon_{t}\varepsilon_{n-t}-g_3(y)f_1(y)\\
 \end{array}
\right|\begin{array}{l}1\<t,t',t_1,
t_2, t_3,  \<n-1\\
{\rm with} \  t+t_1\neq n,\\ t+t_2<n,
t+t_3>n,
\end{array}\right\},$$ where $f_1(y)$, $f_2(y)$ and $ f_4(y)$ are given in the last subsection.
Denote by $J$ the ideal of $\mathbb Z[X']$ generated by $I$. Then we have the following theorem.
\begin{theorem}\label{5.5.10}
The projective ring  $r_p(\widetilde{U}_q)$ is isomorphic to $\mathbb Z[X']/J$.
\end{theorem}
\begin{proof}
By Corollary \ref{5.5.3}, there is a ring epimorphism $\phi$ from $\mathbb Z[X']$ onto $ r_p(\widetilde{U}_q)$ such that
$\phi(y)=[V_2]$, $\phi(x_t)=[{\bf V}(t,0)]$ and $\phi(\varepsilon_t)=[{\bf V}(t,1)]$ for $1\<t\<n-1$. By Proposition \ref{4.2.6}, Corollaries \ref{5.5.8}-\ref{5.5.5} and Proposition \ref{5.5.7}, $\phi(I)=0$. Hence $\phi$ induces a ring epimorphism $\overline{\phi}:\mathbb Z[X']/J\longrightarrow r_p(\widetilde{U}_q)$ such that $\phi=\overline{\phi}\pi$, where $\pi:\mathbb Z[X']\longrightarrow \mathbb Z[X']/J$ is the canonical projection. Let $\overline{u}:=\pi(u)$ for any $u\in \mathbb Z[X]$. Then by the definition of $I$, one knows that $\mathbb Z[X']/J$ is generated, as a $\mathbb Z$-module, by
$$\begin{array}{c}
W=\left\{\overline{y}^i,\overline{y}^j\overline{x}_t,\overline{y}^l\overline{w}_t\left|0\<i\<2n-2,0\<j\<\frac{n-1}{2},
0\<l\<\frac{n-3}{2},1\<t\<n-1\right\}\right..\\
\end{array}$$
Moreover, we have $\phi(\overline{y}^i)=[V_2]^i$, $\phi(\overline{y}^j\overline{x}_t)=[V_2]^j[{\bf V}(t,0)]$ and $\phi(\overline{y}^l
\overline{w}_t)=[V_2]^l[{\bf V}(t,1)]$. By Corollary \ref{5.5.9}, $W$ is a linearly independent set over $\mathbb Z$. It follows that the set $W$ is a $\mathbb Z$-basis of $\mathbb Z[X']/J$. Consequently, $\overline{\phi}$ is a
$\mathbb Z$-module isomorphism, hence a ring isomorphism.
\end{proof}

Now we investigate the Green ring $r(\widetilde{U}_q)$. By Theorem \ref{5.3.10}, we have the following lemma.

\begin{lemma}\label{5.5.11}
$r(\widetilde{U}_q)$ is a free $\mathbb Z$-module with a $\mathbb Z$-basis
$$\left\{\left.
\begin{array}{l}
[V_l], [P_t], [\Omega^{\pm s}V_t],\\

[M_s(t,\eta)], [{\bf V}(t,r)]\\
\end{array}
\right|
\begin{array}{l}
1\<l\<n,1\<t\<n-1,\\
 s\>1, \eta\in \overline{\Bbbk}, r\in\mathbb{Z}_n\\
 \end{array}
 \right\}$$
\end{lemma}

In $r(\widetilde{U}_q)$, let $z_{+}=[\Omega V_1]$, $z_{-}=[\Omega^{-1} V_1]$, $w_{s,\eta}=[M_s(1,\eta)]$ for $s\>1$, $\eta\in \overline{\Bbbk}$ as in the last subsection, and let $y=[V_2]$, $x_t=[{\bf V}(t,0)]$, $\varepsilon_t=[{\bf V}(t,1)]$ as above, where $1\<t\<n-1$.
\begin{proposition}\label{5.5.12}
 $r(\widetilde{U}_q)$ is generated as a ring by $$\{y,x_t,\varepsilon_t,z_+,z_-,w_{s,\eta} |1\<t\<n-1, s\>1,\eta\in \overline{\Bbbk}\}.$$
\end{proposition}
\begin{proof}
Follows from Lemma \ref{5.5.11}, Proposition \ref{4.2.14} and Corollary \ref{5.5.3}.
\end{proof}

\begin{proposition}\label{5.5.13}
 The following set is also a $\mathbb Z$-basis of $r(\widetilde{U}_q)$.\\
 $$\left\{ \left. \begin{array}{l} y^j,y^kx_t, y^i\varepsilon_t,\\
 y^lz^s_+,y^lz^s_-,\\ y^lw_{s,\eta}\end{array}\right|\begin{array}{l}0\<j\<2n-2,0\<k\<\frac{n-1}{2},\\
 0\<i\<\frac{n-3}{2}, 1\<t\<n-1,\\ 0\<l\<n-2, s\>1,\eta\in \overline{\Bbbk}\end{array} \right\}.$$
\end{proposition}
\begin{proof}
Let $N$ be the $\mathbb Z$-submodule of $r(\widetilde{U}_q)$ generated by
$\{[{\bf V}(t,r)]|1\<t\<n-1, r\in\mathbb{Z}_n\}.$ Then $r(\widetilde{U}_q)=N\oplus r(\overline{U}_q)$ as $\mathbb Z$-modules. By Corollary \ref{5.5.9}, $r_p(\widetilde{U}_q)$ has a $\mathbb Z$-basis
$$\begin{array}{c}
\left\{y^j,y^kx_t,y^i\varepsilon_t\left|0\<j\<2n-2,0\<k\<\frac{n-1}{2},
0\<i\<\frac{n-3}{2},1\<t\<n-1\right\}\right.\\
\end{array}$$
By Lemmas \ref{4.2.1} and \ref{5.5.1},  $r_p(\widetilde{U}_q)=N\oplus r_p(\overline{U}_q)$.
Hence it follows from Proposition \ref{5.4.3} and Corollary \ref{4.2.7} that $\{y^kx_t,y^i\varepsilon_t|0\<k\<\frac{n-1}{2},0\<i\<\frac{n-3}{2},1\<t\<n-1\}$ is a $\mathbb Z$-basis of $N$.
Thus, the proposition follows from Proposition \ref{4.2.15}.
\end{proof}

\begin{proposition}\label{5.5.14}
Let $1\<t\<n-1$, $s\>1$ and $\eta\in \overline{\Bbbk}$. Then we have
\begin{enumerate}
\item[$(1)$] $
z_+x_t=z_-x_t=(2g(y)-1)x_t$,
\item[$(2)$] $z_+\varepsilon_t=z_-\varepsilon_t=(2g(y)-1)\varepsilon_t$,
\item[$(3)$] $w_{s,\eta}x_t=w_{s,\eta}\varepsilon_t=sg(y)\varepsilon_t$,
\end{enumerate}
where $g(y)=\sum_{k=1}^{\frac{n-1}{2}}\sum_{i=0}^{[\frac{n-1-2k}{4}]}(-1)^i\frac{k+2i}{k+i}\binom{k+i}{i}y^{k}
+\sum_{1\<i\<[\frac{n-1}{4}]}(-1)^i2+1\in r(\widetilde{U}_q)$.
\end{proposition}
\begin{proof}
Follows from Corollaries \ref{5.4.7}-\ref{5.4.9} and the proof of Corollary \ref{5.5.8}.
\end{proof}
Let $Y=\{ y, z_+, z_-, w_{k,\eta},x_t,\varepsilon_t|k\>1, 1\<t\<n-1,\eta\in\ol{\Bbbk}\}$, and let $\mathbb{Z}[Y]$
be the corresponding polynomial ring. Then $U\subseteq \mathbb{Z}[X]\subseteq \mathbb{Z}[Y]$ and $I\subseteq \mathbb{Z}[X']\subseteq \mathbb{Z}[Y]$, where $U$ and $\mathbb{Z}[X]$ are defined in the last subsection.

Let $U_0$ be a subset of $\mathbb{Z}[Y]$ as follows:
$$U_0=\left\{\left.\begin{array}{l}
  (z_+-z_-)x_t,\\
(z_+-z_-)\varepsilon_t,\\
w_{s,\eta}x_t-w_{s,\eta}\varepsilon_t,\\
z_+x_t-(2g_4(y)-1)x_t,\\
z_+\varepsilon_t-(2g_4(y)-1)\varepsilon_t,\\
w_{s,\eta}x_t-sg_4(y)\varepsilon_t,\\

 \end{array}
\right|\begin{array}{l}
s \>1,\\
1\<t\<n-1,\\
\eta \in\ol{\Bbbk}\\
\end{array}\right\},$$
Let $T$ be the ideal of  $\mathbb{Z}[Y]$ generated by $I\cup U\cup U_0$.

\begin{theorem}\label{5.5.15}
The Green ring $r(\widetilde{U}_q)$ is isomorphic to  $\mathbb{Z}[Y]/T$
\end{theorem}
\begin{proof}
The Green ring $r(\widetilde{U}_q)$ is commutative by \seref{5.4}. Hence by Proposition \ref{5.5.12}, there is a ring epimorphism $\psi$ from $\mathbb{Z}[Y]$ onto  $r(\widetilde{U}_q)$ such that $\psi(y)=[V_2]$, $\psi(z_+)=[\Omega V_1]$, $\psi(z_-)=[\Omega^{-1} V_1]$, $\psi(w_{k,\eta})=[M_k(1,\eta)]$, $\psi(x_t)=[{\bf V}(t,0)]$ and
$\psi(\varepsilon_t)=[{\bf V}(t,1)]$ for any $k\>1$,$1\<t\<n-1$ and $\eta\in \overline{\Bbbk}$. It follows from Theorems \ref{4.2.17}, \ref{5.5.10} and Proposition \ref{5.5.14} that  $\psi(T)=0$. Hence there exists a unique ring epimorphism $\overline{\psi}$ from $\mathbb{Z}[Y]/T$ onto $r(\widetilde{U}_q)$ such that $\overline{\psi}(\overline{u})=\psi(u)$, $u\in \mathbb{Z}[Y]$, where $\overline{u}$ is the image of $u$ under the canonical projection from  $\mathbb{Z}[Y]$ onto  $\mathbb{Z}[Y]/T$. By the definition of $T$, $\mathbb{Z}[Y]/T$ is generated, as a $\mathbb{Z}$-module, by
$$ W=\left\{ \left. \begin{array}{l} \ol{y}^j,\ol{y}^k\ol{x}_t, \ol{y}^i\ol{\varepsilon}_t,\\
 \ol{y}^l\ol{z}^s_+,\ol{y}^l\ol{z}^s_-,\\ \ol{y}^l\ol{w}_{s,\eta}\end{array}\right|\begin{array}{l}0\<j\<2n-2,0\<k\<\frac{n-1}{2},\\
 0\<i\<\frac{n-3}{2}, 1\<t\<n-1,\\ 0\<l\<n-2, s\>1,\eta\in \overline{\Bbbk}\end{array} \right\}.$$
Moreover, we have $\overline{\psi}(\overline{y}^j)=[V_2]^j$, $\overline{\psi}(\overline{y}^{k}\overline{x}_t)=[V_2]^{k}[{\bf V}(t,0)]$,
 $\overline{\psi}(\overline{y}^{i}\overline{\varepsilon}_t)=[V_2]^{i}[{\bf V}(t,1)]$, $\overline{\psi}(\overline{y}^l\overline{z}^s_+)=[V_2]^l[\Omega V_1]^s$,
 $\overline{\psi}(\overline{y}^l\overline{z}^s_-)=[V_2]^l[\Omega^{-1} V_1]^s$ and
 $\overline{\psi}(\overline{y}^l\overline{w}_{s,\eta})=[V_2]^l[M_s(1,\eta)]$.
By Proposition \ref{5.5.13}, $W$ is a linearly independent set over $\mathbb{Z}$, and so $W$ is a $\mathbb Z$-basis of  $\mathbb{Z}[Y]/T$. Consequently, $\overline{\psi}$ is a $\mathbb Z$-module isomorphism, hence a ring isomorphism.
\end{proof}

By Remark \ref{s iso s}, the stable Green ring $r_{st}(\widetilde{U}_q)$ of the small quasi-quantum group $\widetilde{U}_q$ is isomorphic to the stable Green ring $r_{st}(\overline{U}_q)$ of the small quantum group $\overline{U}_q$. It follows from Lemma \ref{4.2.4} that $r_{st}(\overline{U}_q)$ is isomorphic to  $r(\overline{U}_q)/([V_n])$, the quotient ring of $r(\overline{U}_q)$ modulo the ideal $([V_n])$ generated by $[V_n]$. Thus by Theorem \ref{4.2.17}, one obtains the following corollary.
\begin{corollary} Let $\mathbb{Z}[Y']$ be the polynomial ring in following variables:
$$Y'=\{ y, z_+, z_-, w_{k,\eta}|k\>1, \eta\in\ol{\Bbbk}\},$$ and let $J$ be the ideal of  $\mathbb{Z}[Y']$ generated by the following subset:
$$\left\{\begin{array}{l}
  f_1(y), z_+z_--1,\\
(z_+-f_4(y))w_{k,\eta}, (z_+-z_-)w_{k,\eta},\\
w_{k,\eta}w_{s,\a},
w_{k,\eta}(w_{t,\eta}-1-f_4(y))\\
 \end{array}
\left|\begin{array}{l}
k, s, t\>1\\
\mbox{with } k\<t,\\
\eta, \a\in\ol{\Bbbk}\\
\mbox{with } \eta\neq\a\\
\end{array}\right\}\right.,$$
where $f_1(y)$, $f_4(y)\in \mathbb{Z}[y]$ are given in the last subsection. Then the stable Green ring  $r_{st}(\widetilde{U}_q)$ is isomorphic to the quotient ring $\mathbb{Z}[Y']/J$.
\end{corollary}

\centerline{ACKNOWLEDGMENTS}
This work is supported by NNSF of China (Nos. 12071412, 12201545).\\

\end{document}